\documentclass[11pt,a4paper]{amsart}
\usepackage[utf8]{inputenc}
\usepackage{amsmath}
\usepackage{amsfonts}
\usepackage{amsmath}
\usepackage{amsthm}
\usepackage{amssymb}
\usepackage{graphicx}
\usepackage[shortlabels]{enumitem}
\usepackage{float}
\usepackage{url}
\usepackage{hyperref}		
\usepackage{yfonts}   
\usepackage[sorting=nty,style=alphabetic]{biblatex}
\newtheorem{thm}{Theorem}[section]

\newtheorem{lm}[thm]{Lemma}

\theoremstyle{definition}
\newtheorem{df}[thm]{Definition}

\theoremstyle{remark}
\newtheorem{rem}[thm]{Remark}
\newtheorem{clm}[thm]{Claim}

\makeatletter
\newcommand{\xhookdoubleheadrightarrow}[2][]{%
  \lhook\joinrel
  \ext@arrow 0359\rightarrowfill@ {#1}{#2}%
  \mathrel{\mspace{-15mu}}\rightarrow
}
\makeatother
\usepackage[a4paper, left = 2.5cm, right = 2.5cm, top = 2.5cm, bottom = 2.5cm, headsep = 1.2cm]{geometry}
\author{Karol Duda}
\title{Hall's Harem Theorem with controlled sizes of cycles}
\date{}
\begin{document}
\begin{abstract} 
We prove a new version of Hall's Harem theorem, where the final matching is realized by a unary  function with additional conditions on the behavior of cycles. 
\end{abstract}

\maketitle

\section{Introduction} 
The Hall Harem theorem describes a condition which is equivalent to existence of a perfect $(1,k)$-matching of a bipartite graph, for example see Section III.3 of \cite{BB}, Theorem 2.4.2. in \cite{1986C1} or  Theorem H.4.2. in \cite {csc}. 
Assuming that such a graph is of the form $(\mathbb{N},\mathbb{N}, E)$, where $E$ denotes the set of edges between natural numbers, the matching clearly realizes a $k$ to $1$ function, say $f: \mathbb{N} \rightarrow \mathbb{N}$. 
In this paper we study properties which can be 
additionally added to such a function. 
In particular we are especially interested in fast reaching cycles of $f$. 
This will be formalized below as the property of {\em controlled sizes of cycles}. 

The original motivation to such investigation comes from computable amenability, a topic which recently appeared in papers of M. Cavaleri \cite{Cav1, Cav2}, N. Moryakov \cite{mor} and the author (together with A. Ivanov), see \cite{DuI, CPD}. 
%
%
In order to obtain a computable version of Schneider's theorem on existence of $d$-regular forests in non-amenable coarse spaces \cite{schndr} some special form of Hall's harem theorem is required.
In fact we need a matching which is realized by a computable subset of $\mathbb{N}^2$ with the additional property that there is an algorithm which decides whether two natural numbers $m$ and $n$ belong to the same connected component of $f$.  
The property of controlled sizes of cycles which was mentioned above guarantees existence of such an algorithm for computable functions.

However, in this paper we avoid computability issues as much as it is possible. 
There are two reasons for this. 
The first one is the author's conviction that a classical version of Hall's harem theorem with additional restrictions on $f$ is desirable.  
The second reason is the fact that the task of computability of $f$ makes our proofs enormously complicated. 
We have to use very special arguments, which develop the method of Kierstead \cite{hak} and the further modification of it given in \cite{CPD}. 
Thus, the essence of the proof is not easily available. 
The present paper can be considered as a helpful companion of \cite{AC1}, where the computable  version of the main theorem of the present paper is posted. 
Furthermore, it is worth noting here that these two versions of Hall's harem theorem are independent: none of them follows from the other one. 

The formulation of the classical version of Hall's harem theorem with controlled sizes of cycles is given in Theorem \ref{hhco} of Section 2.  
Sections 3 - 6 are devoted to the proof. 

\section{Preliminaries}

\subsection{General preliminaries}
To introduce the reader to the subject we recall Hall's Harem theorem. 
We begin with some necessary definitions from graph theory. 
We mostly follow the notation of \cite {csc}. 
A graph $\Gamma=(\mathbf{\Gamma},E)$ is called a \textit{bipartite graph} if the set of vertices $\mathbf{\Gamma}$ is partitioned into sets $U$ and $V$ in such a way, that the set of edges $E$ is a subset of $U\times V$. 
We denote such a bipartite graph by $\Gamma=(U,V,E)$.

From now on we concentrate on bipartite graphs.  
Note that although the definitions below concern this case, they usually have obvious extensions to all ordinary graphs. 

A subgraph of $\Gamma$ is a triple $\Gamma' = (U',V',E')$ with  $U'\subseteq U$, $V'\subseteq V$, $E' \subseteq E$. 
When $\Gamma'$ is such a subgraph but only the sets of its vertices are specified (i.e. for example $\Gamma'=(U',V')$), this means that $E' = E \cap (U'\times V')$. 
Then we say that the subgraph $\Gamma'$ is induced in $\Gamma$ by the set of vertices. 

Let $\Gamma=(U,V,E)$. 
We will say that an edge $(u,v)$ is \textit{incident} to vertices $u$ and $v$. In this case we say that $u$ and $v$ are \textit{adjacent}. 
When two edges $(u,v),(u',v')\in E$ have a common incident vertex we say that $(u,v),(u',v')$ are also \textit{adjacent}. 
A sequence $x_1,x_2,\ldots, x_n$ is called a \textit{path}, if each pair $x_i,x_{i+1}$ is adjacent, $1\leq i< n$.

Below we will denote the set of vertices $\mathbf{\Gamma}$ by the same letter with the graph as a structure, i.e. $\Gamma$. 
Given a vertex $x\in \Gamma$, the \textit{neighborhood} of $x$ is a set 
$$
N _{\Gamma}(x)=\{y\in \Gamma : (x,y)\in E\}.
$$  
For subsets $X\subset U$ and $Y\subset V$, we define the neighborhood 
$N _{\Gamma}(X)$ of $X$ and the neighborhood $N _{\Gamma}(Y)$ of $Y$ by 
\[ 
N _{\Gamma}(X)=\bigcup\limits_{x\in X} N _{\Gamma}(x) \, \, \text{ and } \,  \, N _{\Gamma}(Y)=\bigcup\limits_{y\in Y} N _{\Gamma}(y).
\]  
We drop the subscript $\Gamma$ if it is clear from the context.

\begin{df}\label{connected}
A subset $X$ of $U$ (resp. of $V$) is called \textit{connected} if for all $x,x'\in X$ there exist a path $x=p_0,p_1,\ldots ,p_k=x'$ in $\Gamma$ such that $p_i\in X\cup N_{\Gamma}(X)$ for all $i\leq k$.
\end{df} 
Note here that this definition concerns only bipartite graphs.
This notion belongs to Kierstead, \cite{hak}.
In the case of connected ordinary graphs we take standard definitions.   

For a given vertex $v$, the \textit{star} of $v$ is the subgraph  
$S=(\{v\}\cup N_{\Gamma}(v),E')$ of $\Gamma$, with \[ 
E'=((\{v\}\cup N_{\Gamma}(v))\times (\{v\}\cup N_{\Gamma}(v)))\cap E. 
\]
A $(1,k)$-\textit{fan} is a subset of $E$ consisting of $k$ edges incident to some vertex $u\in U$. 
We say that $u$ is the \textit{root} of the fan, and when $(u,v)$ belongs to the fan we call $v$ a \textit{leaf} of it.

\begin{df}
\begin{itemize} 
\item An $(1,k)$-\textit{matching} from $U$ to $V$ is a collection $M$ of pairwise disjoint $(1,k)$-fans. 
\item The $(1,k)$-matching $M$ is called \textit{left perfect} (resp. \textit{right perfect}) if each vertex from $U$ is a root of a fan from $M$ (resp.  each vertex from $V$ belongs to exactly one fan of $M$).  
\item The $(1,k)$-matching $M$ is called \textit{perfect} if it is both left and right perfect.
\end{itemize} 
\end{df} 
We often view an $(1,k)$-matching as a bipartite graph $M$ where the fan of $u\in U$ is the $M$-\textit{star} of $u$, i.e. the star of $u$ in the subgraph $M$. 
We emphasize that a perfect $(1,k)$-matching from $U$ to $V$ is a set $M\subset E$ satisfying following conditions: 
\begin{enumerate}[label={(\arabic*)}]
\item for each vertex $u \in U$ there exists exactly $k$ vertices $v_1,\ldots, v_k \in V$ such that \\ 
$(u,v_1),\ldots,(u,v_k)\in M$; 
\item for all $v \in V$ there is a unique vertex $u\in U$ such that $(u,v)\in M$.
\end{enumerate} 

Originally Hall's Marriage Theorem (see e.g. \cite[Theorem 2.1.2]{diestel}, \cite[Section III.3]{BB}, \cite[Theorem 2.4.2.]{1986C1}) provides us a condition for existence of left perfect $(1,1)$-matching in a finite bipartite graph. 
We are interested in so called \textit{Hall's Harem theorem}, which is a generalization of Hall's Marriage theorem to the case of perfect $(1,k)$-matchings for the locally finite infinite graphs. 

One says that graph $\Gamma=(V,E)$ is \textit{locally finite}, if for all vertices $x\in \Gamma$, the neighborhood $N(x)$ is finite. Note that if $\Gamma$ is locally finite, then $N(X)$ is finite for any finite $X\subset V$.

\begin{thm}[Hall's Harem theorem]\cite[Theorem H.4.2.]{csc}
Let $\Gamma=(U,V,E)$ be a locally finite graph and let $k\in \mathbb{N},\; k\geq 1$. The following conditions are equivalent: 
\begin{enumerate}[label={(\roman*)}]
\item For all finite subsets $X\subset U$, $Y\subset V$ the following inequalities holds: $|N(X)|\geq k|X|$, $|N(Y)|\geq \frac{1}{k}|Y|$.
\item $\Gamma$ has a perfect $(1,k)$-matching.
\end{enumerate} 
\end{thm} 

The first condition in this formulation is known as \textit{Hall's $k$-harem condition}. 

It is a crucial fact that the theorem  holds for locally finite infinite graphs.
At this point we inform the reader that the statement of the Hall Marriage Theorem does not work for infinite graphs in general (see e.g. \cite[S.2 Ex.6]{diestel}).  Nevertheless, there are versions of this theorem for graphs of any cardinality \cite{ahr}. 

\subsection{Reflections} 
Throughout the paper, $d$ is a natural number greater than $1$.
When $\Gamma=(U,V,E)$ is a bipartite graph, we always assume that $V\subseteq U\subseteq\mathbb{N}$, i.e. $V$ is a subset of the right copy of $U$. 

The following notation substantially simplifies the presentation.
For every vertex $v\in V$ there exist a vertex from $U$ which is a {\em copy} of $v$ (i.e. the same natural number), we denote it by $u_v$.  
If a vertex $u\in U$ has the copy in $V$ then we denote this copy by $v_u$. 

\begin{df} 
The graph $\Gamma=(U,V,E)$ is called $U$-\textit{reflected} if $V$ is a subset of the right copy of $U$ and for every edge $(u,v)\in E$ with $v_{u}\in V$ the edge $(u_{v},v_{u})$ is also in $E$. 
If additionally $V$ is the right copy of $U$, then $\Gamma$ is called a \textit{fully reflected} bipartite graph. 
\end{df}
\noindent 
It is worth noting that when $\Gamma' = (U',V')$ is an induced subgraph of an $U$-reflected  $\Gamma=(U,V,E)$ such that $V'$ is a subset of the right copy of $U'$ then $\Gamma'$ is $U'$-reflected.

The main theorem below states that in the case of fully reflected bipartite graphs the Hall's harem condition allows us to force some additional properties at the expense of obtaining a perfect $(1,(d-1))$-matching instead of a $(1,d)$-matching. 
We will now give necessary details.  

\subsection{Controlled sizes of cycles. Main theorem}
Let $f$ be a function. If for some $i\neq 0$ we have $f^i(u)=u$ then we will say that $u$ is a \textit{periodic point} of $f$. 
For such $u$ and the smallest $i\neq 0$ with $f^i(u)=u$, we say that $\{u,f(u),\ldots, f^{i-1}(u)\}$ is a {\em cycle} of $f$.
Any $(1,(d-1))$ matching can be considered as a $(d-1)$ to $1$ function $f:\mathbb{N}\rightarrow \mathbb{N}$. Moreover, if this matching is perfect, such a function $f$ is total and surjective.
We roughly want to show that given a fully reflected bipartite graph satisfying Hall's $d$-harem condition, there is a perfect $(1,(d-1))$ matching $f:\mathbb{N}\rightarrow\mathbb{N}$, such that for each $u$ there exist $i\geq 0$ such that $f^i(u)$ is a periodic point.

\begin{df}\label{cycles}
Let $f:\mathbb{N}\rightarrow \mathbb{N}$ be a $(d-1)$ to $1$ function. 
We say that $f$ has \textit{controlled sizes of its cycles} if each of the following conditions holds:  
\begin{enumerate}[(i)]
\item $f^2(1)=1$;
\item if $n\geq 2$ and $f^i(n)=n$ then $i\leq n $;
\item if $n\geq 2$ and for all $i\leq n$ we have $f^i(n)\neq n$ then there exist $k\leq 2n$ and $l\leq n$ such that $f^{k+l}(n)=f^k(n)$.
\end{enumerate}
\end{df}  
The following theorem is the main result of the paper. 

\begin{thm}\label{hhco} (Main Theorem) 
Let $\Gamma=(U,V,E)$ be a locally finite bipartite graph such that: 
\begin{itemize}
\item both $U$ and $V$ are identified with $\mathbb{N}\setminus\{0\}$, 
\item $E$ does not contain edges of the form $(u,v_u)$,  
\item  $\Gamma$ is fully reflected,  
\item $\Gamma$ satisfies Hall's $d$-harem condition.
\end{itemize}
Then there exists a perfect $(1,d-1)$-matching of $\Gamma$, which realizes a $(d-1)$ to $1$ function $f:\mathbb{N}\rightarrow \mathbb{N}$ with controlled sizes of its cycles.
\end{thm} 

The proof of the theorem is based on an inductive construction of the matching. 

\subsection{Notation used in the construction} 
Since this construction is highly technical, 
we start with a list of the notation. 
We do not insist on a thorough inspection of it. 
In the beginning a hasty view will suffice. 
\begin{itemize}

\item $M$ is a perfect matching that we construct.

\item $M_{n-1}$ is a set of $(1,d-1)$-fans added to $M$ at the end of the $n$-th step. Thus $M=\bigcup\limits_{n=1}^{\infty}M_{n-1}$.

\item $\Gamma^{(0)}=(U^{(0)},V^{(0)},E^{(0)})$ is the original graph $\Gamma$.

\item $U^{(n)}:=U^{(n-1)}\setminus \{u\in U^{(n-1)}: \exists v\in V^{(n-1)}, (u,v)\in M_{n-1}\}$.

\item $V^{(n)}:=V^{(n-1)}\setminus \{v\in V^{(n-1)}: \exists u\in U^{(n-1)}, (u,v)\in M_{n-1}\}$.

\item $\Gamma^{(n)}=(U^{(n)},V^{(n)})$. 
We will see that $\Gamma^{(n)}$ is $U^{(n)}$-reflected.

\item After the $n$-th step we obtain decompositions 
$U^{(n)}= U^{(n)\star}\, \dot{\cup} \, U^{(n)\perp}$ and  $V^{(n)}= V^{(n)\star}\, \dot{\cup} \, V^{(n)\perp}$,  
where we say that $U^{(n)\perp}$ consists of elements from $U^{(n)}$ which might spoil Hall's $d$-harem condition for $\Gamma^{(n)}$. 

\item Put $U^{(0)\perp} =\emptyset$ and $V^{(0)\perp}=\emptyset$ 
(since $\Gamma^{(0)}$ is $\Gamma$, i.e. it satisfies Hall's $d$-harem condition). 

\item $\Gamma^{(n)\perp}$ is a graph with the sets of vertices $(U^{(n)\perp},V^{(n)\perp})$ and the set of edges corresponding to $(1,d-1)$-fans with roots denoted by $u^{\perp}_i\in U^{(n)\perp}$. 

\item When $U^{(n)\perp}\setminus U^{(n-1)\perp}$ is not empty,  $U^{(n)\perp}\setminus U^{(n-1)\perp}=\{u^{\perp}_{n-1}\}$ and $V^{(n)\perp}\setminus V^{(n-1)\perp}$ consists of leaves $\{v^{\perp}_{n-1,i}:1\leq i\leq d-1\}$. 

\item  $\Gamma^{(n)\star}:=(U^{(n)}\setminus U^{(n)\perp}, V^{(n)}\setminus V^{(n)\perp})$. 
We will see that $\Gamma^{(n)\star}$ satisfies Hall's $d$-harem condition.

\item During the construction of $M_{n}$ we will define fans $M^j_{n}$, $j\leq n+1$. 
The graph $M_{n}$ is the union of them. 

\item $M^j_{n}$ is a fan consisting of edges denoted by $\{(u^{j}_n,v^j_{n,i}): i\leq d-1\}$.

\item $u_n$ is a starting vertex of the $n$-th step, it is also denoted by $u^0_n$.

\item For any subgraph $\Gamma'$ of $\Gamma^{(n)}$ we denote by $\Gamma'(-u^{0}_n,\ldots, -u^{j}_n)$ the graph obtained from $\Gamma'$ by removal of the $(1,d-1)$-fans of $M_n$ with roots $u^{0}_n,\ldots, u^{j}_n$. 

\item For any subgraph $\Gamma'=(U',V')$ of $\Gamma^{(n)}$ and any $u_j^{\perp}\in U^{(n)\perp}$ we denote by $\Gamma'(+u_j^{\perp})$ the graph induced in $\Gamma^{(n)}$ by the sets of vertices $U'\cup\{u_j^{\perp}\}$ and $V'\cup\{v^{\perp}_{j,i}: 1\leq i\leq d-1\}$. 

\item For any subgraph $\Gamma'=(U',V')$ of $\Gamma^{(n)}$ and any vertex $v\in V^{(n)}$ we denote by $\Gamma'(+v)$ (resp. $\Gamma'(-v)$) the graph induced in $\Gamma^{(n)}$ by the sets of vertices $U'$ and $V'\cup\{v\}$ (resp. $V'\setminus \{ v\}$). 

\item $\mathfrak{M}^1_{n}$ denotes a perfect $(1,d)$-matching in $\Gamma^{(n)\star}$ which appears in the first part of step $n+1$. 

\item $\mathfrak{M}^2_{n}$ denotes a perfect $(1,d)$-matching in $\Gamma^{(n)}(-u^{0}_n,\ldots, -u^{j}_n)\cap \Gamma^{(n)\star}$ for some $j$, which appears in the second part of step $n+1$.

\item The elements adjacent to $u^{j+1}_n$ in the matching $\mathfrak{M}^2_{n}$ are denoted by $\dot{v}^{j+1}_{n,1}\ldots  \dot{v}^{j+1}_{n,d}$. 
The element $\dot{v}^{j+1}_{n,1}$ is a candidate for $v^{j+1}_{n,1}$. 

\item The fan $(u^{\perp}_{n}, \{ v^{\perp}_{n,j} \, |1\leq j\leq d-1 \})$ usually appears as a part of a fan of the matching $\mathfrak{M}^2_{n}$. 
We warn the reader that it is possible that $u^{\perp}_{n}$ does not exist.   

\item We assume that all of these elements (i.e. $u_n, v^j_{n,i}, u^j_n , v^{\perp}_{j,i}, \ldots$ ) are natural numbers and are ordered according to the standard ordering of natural numbers.
\end{itemize}

Before describing the construction in detail, we show the single step of the algorithm in the picture attached below, see Figure 1. 
We inform the reader that detailed illustrations of all parts of the step will be given in the appropriate places.

\begin{figure}[ht]
  \centering
    \includegraphics[width=1\textwidth]{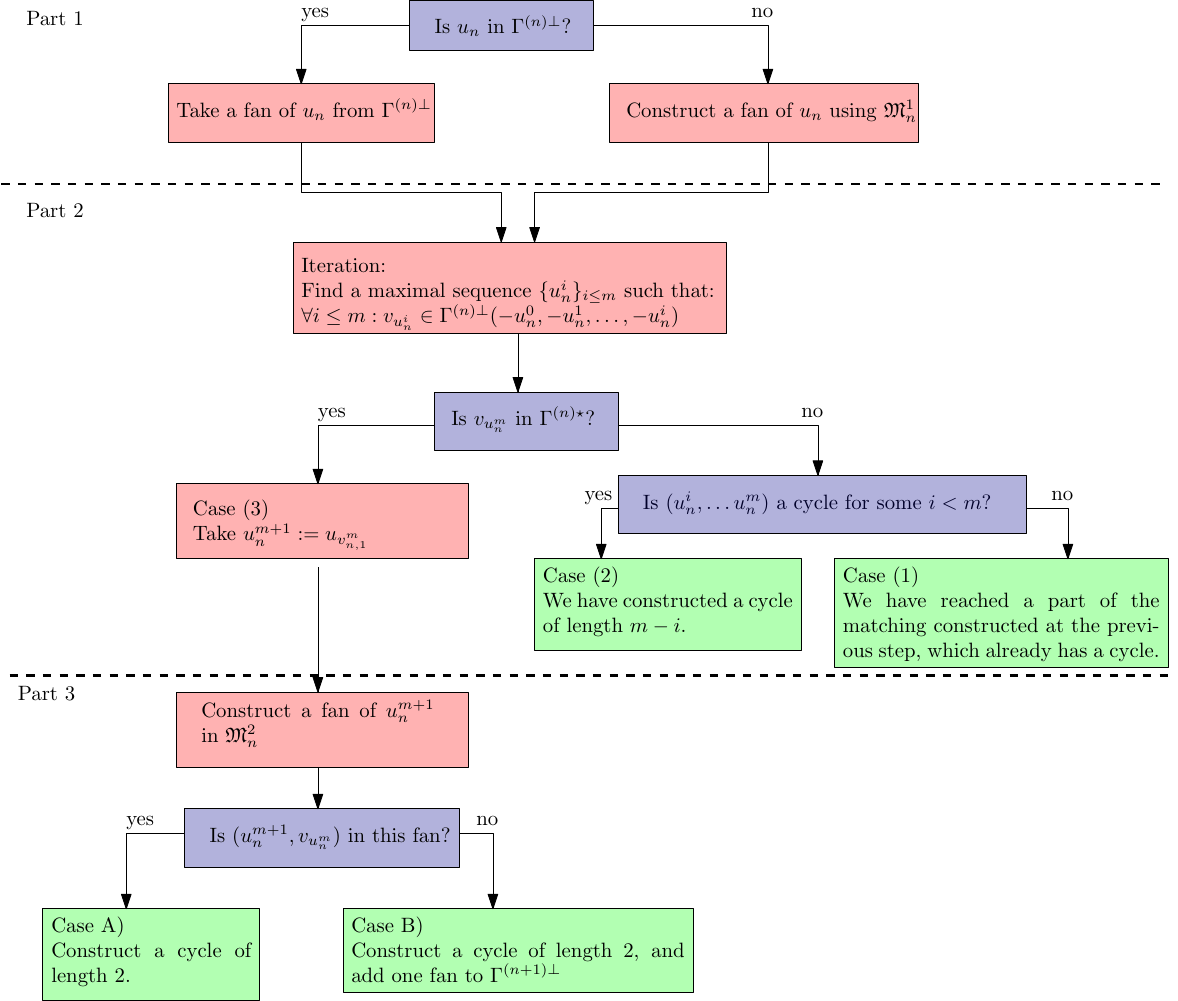}
    	\caption{Single step.}
\end{figure}

\section{The construction}\label{simconst}
We assume that $\Gamma=(U,V,E)$ is a bipartite graph satisfying Hall's $d$-harem condition, such that:
\begin{itemize}
\item both $U$ and $V$ are identified with $\mathbb{N}\setminus\{0\}$;
\item $\Gamma$ is fully reflected;   
\item $E$ does not contain edges of the form $(v,u_v)$. 
\end{itemize}

We now describe an inductive construction that is the heart of the 
proof of Theorem \ref{hhco}. 
Every step consists of three parts. 
Some of them are finished by claims stating that certain graphs satisfy Hall's $d$-harem condition. 
These statements support further stages of the construction. 
We start with a detailed description of the first step of the construction.

\subsection{Step 1, part 1} 
We take $u_0$, the first element of the set $U$ (it is clear that $u_0=1$). 
Using Hall's harem theorem, we find a perfect $(1,d)$-matching $\mathfrak{M}^1_{0}$. 
Let $v^0_{0,1},\ldots, v^0_{0,d}$ be elements of $V$, ordered by its numbers, such that $(u_0,v^{0}_{0,i})\in \mathfrak{M}^1_{0}$ for all $i\leq d$. 
Define the fan $M^0_0$ as the set of edges $(u_0,v_{0,i})$ for $i\leq d-1$. 
Let $\Gamma^{(0)}(-u_0):=(U\setminus\{u_0\},V\setminus\{v^0_{0,1},\ldots, v^0_{0,d-1}\})$. 

Note that the graph 
$\Gamma^{(0)}(-u_0)\setminus \{ v^0_{0,d}\}$
satisfies Hall's $d$-harem condition, since it obviously has a perfect $(1,d)$-matching. 
Furthermore, the following (more complicated) statement holds too. 

\begin{clm}\label{cs1p1}
The graph $\Gamma^{(0)}(-u_0)$ satisfies Hall's $d$-harem condition.
\end{clm}
See Lemma \ref{1stsim} for the proof of this claim.

\begin{figure}[h]
  \centering
    \includegraphics[width=0.7\textwidth]{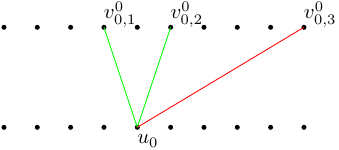}
	\caption{The first part of Step 1, the  $\mathfrak{M}^1_{0}$-fan of $u_0$ where $M^0_0$ is green.}
\end{figure}

\subsection{Remarks before part 2}
Before the second part of step 1 let us discuss our local goals.  
Let $f_0$ be the partial function which corresponds to $M_0$ and let $\Gamma^{(1)}$ be the graph obtained after step 1, i.e. the part of $\Gamma$ after removal $M_0$. 
We want to force that:
\begin{enumerate}
\item for all $n \in \mathsf{Dom}(f_0)$  there exists $i$ such that $f_0^i(n)$ is a periodic point, and
\item $\Gamma^{(1)}$ is $U^{(1)}$-reflected.
\end{enumerate}
It is clear that (1)-(2) are satisfied if we add the edge $(u_{v^0_{0,1}},v_{u_0})$ to the matching. 
This means that $f(v_{u_0}) = u_{v^0_{0,1}}$, i.e.  $f_0^2(u_0)= f_0 (f_0 (v_{u_0})) = u_0$. 

\subsection{Step 1, part 2} 
Denote $u^{1}_0:=u_{v^0_{0,1}}$.
Note that $(u^1_0,v_{u_0})\in\Gamma^{(0)}(-u_0)$, because $(u_0,v^0_{0,1})$ is in $\Gamma^{(0)}$ and the latter one is $U^{(0)}$-reflected.

\begin{rem}
For arbitrary $n$ part 2 of step $n$ depends on the subgraph $\Gamma^{(n-1)\perp}$. Since $\Gamma^{(0)\perp}$ is empty, at step 0 this part is reduced just to the choice of the vertex $u^{1}_0$. 
\end{rem}
In part 3 we will add the edge $(u^{1}_0, v_{u_0})$, to $M_0^1$. 

\subsection{Step 1, part 3} 
Since $\Gamma^{(0)}(-u_0)$ satisfies Hall's $d$-harem condition, it has a perfect $(1,d)$-matching $\mathfrak{M}^2_{0}$. 
We remind the reader that $\dot{v}^1_{0,1},\ldots, \dot{v}^1_{1,d}$ are elements of $V$, ordered by its numbers, such that $(u^1_0,\dot{v}^1_{0,i})\in \mathfrak{M}^2_{0}$ for all $i\leq d$. 
Since $v_{u_0}$ is the least number in $V$, there are two possible cases: 
\begin{enumerate}

\item $v_{u_0}=\dot{v}^1_{0,1}$. 
We set $v^1_{0,i}:=\dot{v}^1_{0,i}$, $1 \le i \le d-1$ (i.e. $v^1_{0,1} = v_{u_0}$). 

\item $v_{u_0}\neq \dot{v}^1_{0,1}$. 
Then find the fan $(u,v_i)\in \mathfrak{M}^2_{0}$, $1\leq i\leq d$, such that $v_{u_0}=v_1$ (assuming that the ordering of $v_i$ corresponds to their indexes). 
We set:
\begin{itemize}
\item $v^1_{0,1}:=v_{u_0}$ and $v^1_{0,i}:=\dot{v}^1_{0,i-1}$ for $2< i\le d-1$,   
\end{itemize} 
and define a candidate for $\Gamma^{(1)\perp}$: 
\begin{itemize} 
\item $\dot{u}^{\perp}_0:=u$;
\item $\dot{v}^{\perp}_{0,i-1}:=v_i$ for $1 < i\le d-1$.
\end{itemize}
\end{enumerate}

In either case define the fan $M^1_0$ as the set of edges $(u^1_0,v^1_{0,i})$ for $1\leq i\leq d-1$.

\begin{figure}[H]
  \centering
    \includegraphics[width=0.65\textwidth]{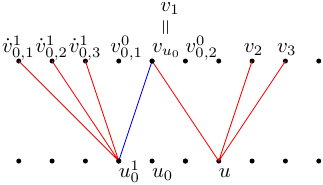}
    \caption{Step 1, part 3. $\mathfrak{M}^2_{0}$ is red and $v_{u_0}$ is matched with $u$. We want the edge $(u^{1}_0, v_{u_0})$ to be in $M_0$. Force the situation from Figure 4.}
\end{figure}

\begin{figure}[H]
  \centering
    \includegraphics[width=0.65\textwidth]{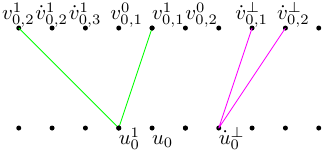}
    \caption{$M^1_{0}$ is green. It is possible that the purple fan consisting of edges $(\dot{u}^{\perp}_0,\dot{v}^{\perp}_{0,1}),(\dot{u}^{\perp}_0,\dot{v}^{\perp}_{0,2})$ will be added to $\Gamma^{(1)\perp}$.}
\end{figure}

Put $M_0=M^0_0\cup M^1_0$. 
We obtain $\Gamma^{(1)}$ by removal of $M_0$ from $\Gamma$. Since $u_0, v_{u_0}, v^0_{0,1}$ and $u^1_0=u_{v^0_{0,1}}$ have been removed, $\Gamma^{(1)}$ is $U^{(1)}$-reflected. 
It might turn out that it does not satisfy Hall's $d$-harem condition. 

Let $\Gamma'_0=(U^{(1)}\setminus \{\dot{u}_0^{\perp}\},V^{(1)}\setminus \{\dot{v}^{\perp}_{0,1},...,\dot{v}^{\perp}_{0,d-1}\})$. 
The following claim follows from Lemma \ref{2ndsim}.

\begin{clm}\label{cs1p2part1}
At least one of 
$ \, \Gamma'_{0} \, $ or $ \, \Gamma^{(1)} \, $  
satisfies Hall's $d$-harem condition.
\end{clm} 

\subsection{The output of the first step}

If $\Gamma^{(1)}$ satisfies Hall's $d$-harem condition, set
$\Gamma^{(1)\star}:=\Gamma^{(1)}$, $U^{(1)\perp}=\emptyset$ and $V^{(1)\perp}=\emptyset$.
If $\Gamma^{(1)}$ does not satisfy Hall's $d$-harem condition, set
$\Gamma^{(1)\star}:=\Gamma'_{0}$ and 
\begin{itemize}
\item $u^{\perp}_0:=\dot{u}^{\perp}_0$;
\item $v^{\perp}_{0,i}:=\dot{v}^{\perp}_{0,i}$, $1 \leq i\leq d-1$;
\item $U^{(1)\perp}=\{u^{\perp}_0\}$;
\item $V^{(1)\perp}=\{v^{\perp}_{0,i}: 1 \leq i\leq d-1\}.$
\end{itemize}

\subsection{The situation before step $n+1$}
At the previous step we constructed graphs $\Gamma^{(n)}$ and $\Gamma^{(n)\star}$, where $\Gamma^{(n)}$ is $U^{(n)}$-reflected and $\Gamma^{(n)\star}$ satisfies Hall's $d$-harem condition.
Since $|U^{(n)\perp}\setminus U^{(n-1)\perp}|\leq 1$, there are at most $n$ roots $u^{\perp}_i$ of fans in $U^{(n)\perp}$ (see Section 2.4 for the corresponding definition).

\subsection{Step n+1, part 1}  
Let $\mathfrak{M}^1_n$ be an $(1,d)$-matching in $\Gamma^{(n)\star}$.  
Take $u_n$, the first element of the set $U^{(n)}$. 
In order to define $M^0_n$ we have two possible cases:

\begin{enumerate} 
\item There is $j$ such that $u_n=u_j^{\perp}\in U^{(n)\perp}$.  Then we set $M^0_n$ to consist of all edges of the form $(u^{\perp}_j,v^{\perp}_{j,i})$ and remove the fan with the root $u^{\perp}_j$ from $\Gamma^{(n)\perp}$. We redefine $U^{(n)\perp}$ and $V^{(n)\perp}$ accordingly (in particular $u_j^{\perp}$ is removed from $U^{(n)\perp}$).

\item If $u_n\not\in U^{(n)\perp}$, then 
verify whether there is $j$ such that $(u_n, v^0_{n,j}) \in \mathfrak{M}^1_n$ and $(u_{v^0_{n,j}},v_{u_n})\in \Gamma^{(n)\perp}$.
By the definition of $\Gamma^{(n)\perp}$ it can happen for at most one $j$.\footnote{Note here that it can also happen that $v_{u_n}$ is not even in $\Gamma^{(n)}$.} 
If there is such $j$, then it can be equal to $d$. In this case, we set $M^0_n$ to consist of edges $(u_n,v^0_{n,i})\in\mathfrak{M}^1_n$ for $i\neq d-1$.
In all other possibilities, we set $M^0_n$ to consist of edges $(u_n,v^0_{n,i})\in \mathfrak{M}^1_n$ for $i\leq d-1$.
\end{enumerate}
\begin{rem} 
Note that as a result in the case of existence of $j$ as in (2) the edge $(u_n, v^0_{n,j})$ is included into $M^0_n$. 
Although the final part of the algorithm of (2) can be presented easier, i.e. as in Figure \ref{fig:part1},   
we do not want to miss this point. 
\end{rem}

\begin{figure}[ht]
  \centering
    \includegraphics[width=1\textwidth]{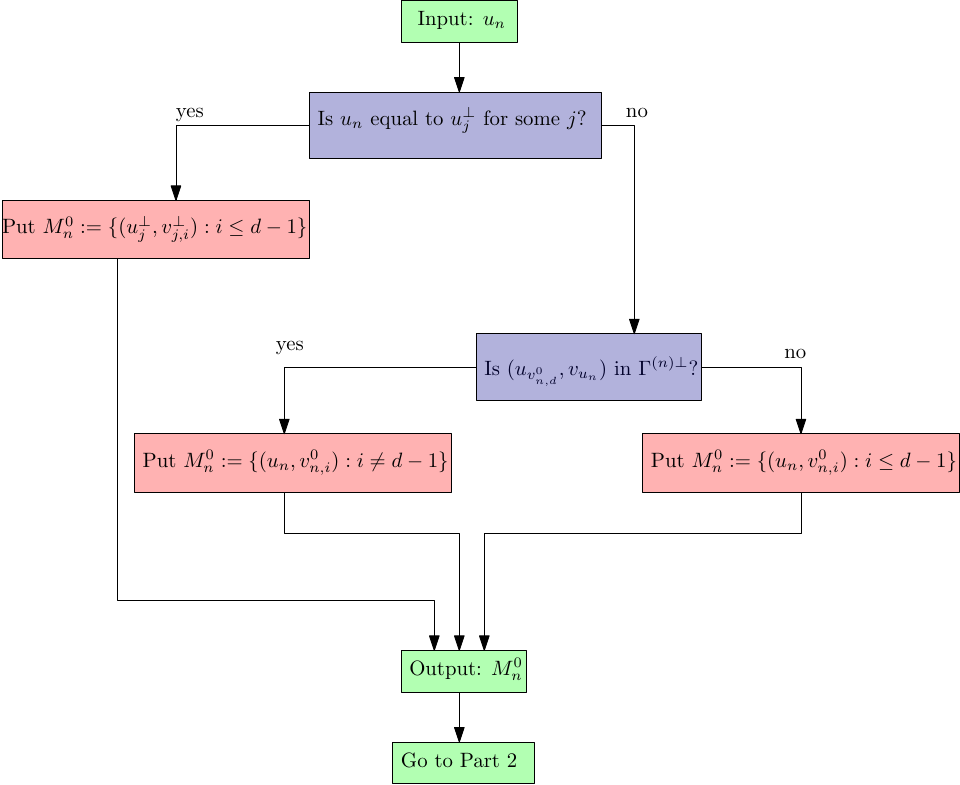}
\caption{Detailed version of the algorithm used at part 1 of step $n+1$.}\label{fig:part1}
\end{figure}

\pagebreak

\begin{clm}\label{csnp1}
Let $\Gamma^{(n)\star}(-u_n)$ be $\Gamma^{(n)}(-u_n)\cap \Gamma^{(n)\star}$.
One of the following holds:
\begin{itemize}
\item $\Gamma^{(n)\star}(-u_n)$ satisfies Hall's $d$-harem condition;
\item there is some vertex $u_j^{\perp}\in U^{(n)\perp}$ such that the graph $\Gamma^{(n)\star}(-u_n,+u_j^{\perp})$ satisfies Hall's $d$-harem condition.
\end{itemize}
\end{clm}

In case (1),  $\Gamma^{(n)\star}(-u_n) =  \Gamma^{(n)\star}$ and the claim is obvious. In case (2) the claim follows from Lemma \ref{1stsim} below.

\subsection{The output of part 1} 

This is $M_n^0$. We also update our graphs in the following situation.
If $\Gamma^{(n)\star}(-u_n)$ does not satisfy Hall's $d$-harem condition, let $u_j^{\perp}$ be an element from $U^{(n)\perp}$ realizing the second possibility of Claim \ref{csnp1}. 
We remove the fan of $u^{\perp}_j$ with its leaves 
from $(U^{(n)\perp},V^{(n)\perp})$ and then we put it into $\Gamma^{(n)\star}$.
Thus the latter graph (and $U^{(n)\perp},V^{(n)\perp}$) are updated.
It is clear that now the redefined $\Gamma^{(n)\star}(-u_n)$ satisfies Hall's $d$-harem condition.

\bigskip 

\subsection{Remarks before part 2} \label{Rbp2} 
Before the rest of step $n+1$, we describe the goals which we want to achieve after the step:

\begin{enumerate}
\item  the partial function $f_n$ corresponding to $\bigcup\limits_{i=0}^{n} M_i$ has controlled sizes of its cycles, and
\item the graph $\Gamma^{(n+1)}$ obtained at the end of the step is $U^{(n+1)}$-reflected and the corresponding graph $\Gamma^{(n+1)\star}$ satisfies Hall's $d$-harem condition.
\end{enumerate}
In order to achieve the first condition we will organize one of the the following properties:
\begin{enumerate}[(i)]
\item there is a sequence of vertices $u^0_n,u^1_n, u^2_n,\ldots, u^j_n$, $1\leq j\leq n$, such that every edge $(u^i_n,v_{u^{i-1}_n})$ belongs to $M_n$, and for some $0\leq \ell \leq j-1$ the sequence $(u^{\ell}_n, u^{\ell +1}_n,\ldots, u^j_n)$ is a cycle;
\item there is a sequence of vertices $u^0_n, u^1_n, u^2_n,\ldots, u^j_n$, $j\leq n$, such that every edge $(u^i_n,v_{u^{i-1}_n})$ belongs to $M_n$, and $v_{u^j_n}$ is already adjacent to some edge from $\bigcup\limits_{i=0}^{n-1} M_i$;
\end{enumerate} 

\subsection{Step n+1, part 2}\label{nplus1p2}
We begin by checking whether $v_{u_n}$ belongs to $\Gamma^{(n)\perp}$. 
If $v_{u_n}\in \Gamma^{(n)\perp}$ then we denote $u_n$ by $u^0_n$ and begin the following process of choosing the consecutive vertices $u^i_n$. 

\begin{quote}
\textit{First step of iteration.} 
Assume that for some $j_0,i$ we have $v_{u_n}=v^{\perp}_{j_0,i}\in V^{(n)\perp}$.
Then we set 

$u^1_n= u^{\perp}_{j_0}$, $v^1_{n,k}:=v^{\perp}_{j_0,k}$, $k\le d-1$, 

$M^1_n:=\{(u^1_n,v^1_{n,k}): 1 \le k \le d-1 
\}$ \\ 
and check whether $v_{u^1_n}\in \Gamma^{(n)\perp}(-u^0_n)$. \\ 
If it is so we repeat the iteration for $v_{u^1_n}$. 
Note that $v_{u^1_n} \in \Gamma^{(n)\perp}(-u^0_n,-u^1_n )$
then. 
\end{quote} 

\bigskip 

\begin{quote}
\textit{Single step of iteration.} 
We verify if 
$v_{u^m_n}\in \Gamma^{(n)\perp}(-u^0_n,-u^1_n,\ldots, -u^{m}_n)$. 
If it is so then for some $j_{m},i$ we have $v_{u^m_n}=v^{\perp}_{j_{m},i}\in V^{(n)\perp}$. 
Define  

$u^{m+1}_n= u^{\perp}_{j_{m}}$,  $v^{m+1}_{n,k}:=v^{\perp}_{j_{m},k}$, $1\le k \le d-1$, 

$M^{m+1}_n:=\{(u^{m+1}_n,v^{m+1}_{n,k}): 1 \le k \le d-1 
\}$,  

and we repeat the iteration for $v_{u^{m+1}_n}$.
This ends the single iteration step.
\end{quote}

Since $|U^{(n)\perp}|\leq n$, the procedure ends after at most $n$ iterations. 
Therefore, one of the following cases is realized for some $l\leq n$: 

\begin{enumerate}
\item $v_{u^l_n}\notin \Gamma^{(n)}$;
\item $v_{u^l_n}\in \Gamma^{(n)}$, but $v_{u^l_n}\notin \Gamma^{(n)}(-u^0_n,-u^1_n,\ldots, -u^{l}_n)$ (this case is impossible for $l=0$);
\item $v_{u^l_n} \in \Gamma^{(n)\star}$ (and obviously $v_{u^l_n}\in \Gamma^{(n)}(-u^0_n,-u^1_n,\ldots, -u^{l}_n)$).
\end{enumerate}

In case (1) $v_{u^l_n}$ was already added to $M$ at preceding steps and condition (ii) described before this stage is satisfied.
We finish part 2 of step $n+1$ without a new cycle.

In case (2), the last iteration closes the cycle $(u^k_n,\ldots u^l_n)$. 
The length of this cycle is not greater than $l+1$.

In each of these two cases we skip part 3 of the step and go to the output of step $n + 1$, see
3.13. 
We set 
$M_n = \bigcup\limits^{l}_{k=0} M^k_n$ and obtain the graph $\Gamma^{(n+1)}$ from $\Gamma^{(n)}$ by removal of $M_n$-fans. 
Since for every $u \in U^{(n)} \setminus U^{(n+1)}$ the element $v_u$ is either removed as well or was not in $V^{(n)}$ from the
beginning, then the graph is $U^{(n+1)}$-reflected.

Case (3) is the most complicated one; part 3 of the step will be entirely dedicated to it. 
In this case, we finish part 2 of the step by taking a new term $u^{l+1}_n := u_{v^l_{n,1}}$. 
In part 3 we will force the cycle $(u^{l}_n,u^{l+1}_n)$ of length $2$.

\bigskip 

Figure \ref{fig:part2} illustrates the algorithm of the second part of the step.

\begin{figure}[ht]
  \centering
    \includegraphics[width=1\textwidth]{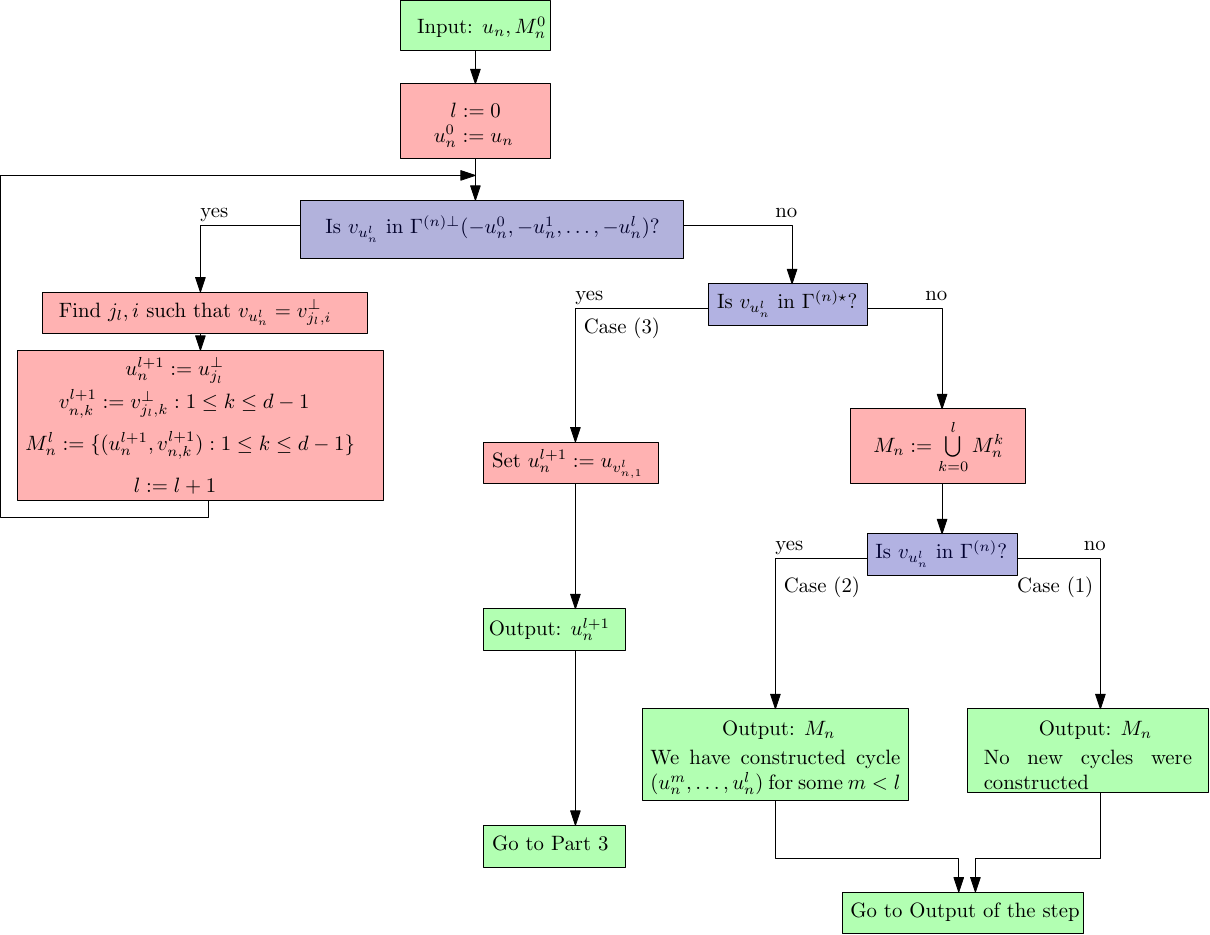}
    \caption{Detailed version of the algorithm of part 2 of step $n+1$.}\label{fig:part2}
\end{figure}

\pagebreak

\subsection{Example}  \label{cycllemma}
Before we move to the third part of the step, we give an example of a cycle obtained by the algorithm in case (2) of this part of the step. The following pictures show how a cycle of length $3$ arises in this procedure when the graph satisfies Hall's $3$-harem condition.

\begin{figure}[H]
  \centering
    \includegraphics[width=0.8\textwidth]{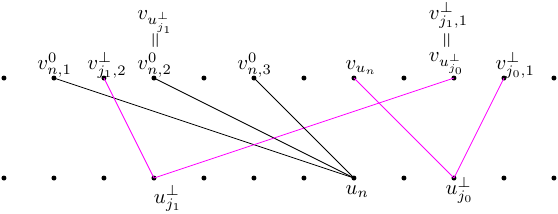}
    \caption{$\Gamma^{(n)\star}$ is black, $\Gamma^{(n)\perp}$ is purple.}
\end{figure}

\begin{figure}[H]
  \centering
    \includegraphics[width=0.8\textwidth]{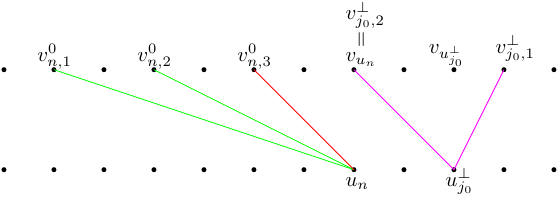}
    \caption{$\mathfrak{M}^n_1$ is red and green, $M_n^0$ is green. We have $v_{u_n}=v^{\perp}_{j_0,2}$.}
\end{figure}

\begin{figure}[H]
  \centering
    \includegraphics[width=0.8\textwidth]{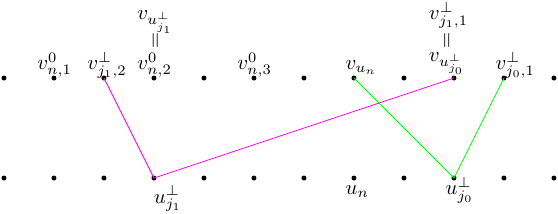}
    \caption{$M_n^1$ is green. We have $v_{u^{\perp}_{j_1}}=v^{\perp}_{j_2,2}$. Moreover $v_{u_{j_1}^{\perp}}=v^0_{n,2}$.}
\end{figure}

To demonstrate the cycle in a single picture, we take 4 copies of $\mathbb{N}$ as rows. 
The first and the second ones correspond to  $V$ and $U$ after part 1 of the step. 
The second and the third rows correspond to the sets $V$ and $U$ after the first iteration step. 
The third and the fourth rows correspond to $U$ and $V$ after the second step of the iteration.

\begin{figure}[H]
  \centering
    \includegraphics[width=0.8\textwidth]{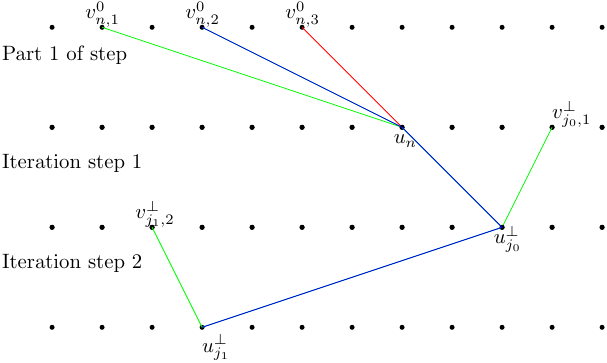}
    \caption{The final $M_n$ with the obtained cycle of length $3$ marked in blue. The red edge is the edge from $\mathfrak{M}^n_1$ that was not added to $M_n^0$.}
\end{figure}

\subsection{Step $n+1$, part 3. }  
Case (3) of part 2 guarantees that the edge $(u^l_n,v^l_{n,1})$ is in $\Gamma^{(n)}$ 
and $v_{u^l_n} \in \Gamma^{(n)}(-u_n,-u^1_n,\ldots,-u^l_n)$. 
Thus, applying $U^{(n)}$-reflectedness of $\Gamma^{(n)}$ we see $(u^{l+1}_n,v_{u^l_n})\in\Gamma^{(n)}(-u_n,-u^1_n,\ldots,-u^l_n)$. 
Observe that since in part 2 the elements $u^{\perp}_{j_1},\ldots, u^{\perp}_{j_l}$ were chosen in $U^{(n)\perp}$, 
$$
\Gamma^{(n)}(-u_n,-u^1_n,\ldots,-u^l_n) \cap\Gamma^{(n)\star}=\Gamma^{(n)}(-u_n)\cap\Gamma^{(n)\star}=\Gamma^{(n)\star}(-u_n).
$$ 
By part 1 of this step, the graph $\Gamma^{(n)\star}(-u_n)$ satisfies Hall's $d$-harem condition. 
Thus, we can find a perfect $(1,d)$-matching $\mathfrak{M}^2_n$ in $\Gamma^{(n)\star}(-u_n)$.
Let us fix it. 

\begin{quotation} 
We now check whether there is $j$ with $u^{l+1}_n=u^{\perp}_j$. 
\\ 
If it exists, we set 
$\dot{v}^{l+1}_{n,i}:=v^{\perp}_{j,i}$ for $1\leq i\leq d-1$. \\ 
If there is no $j$ such that $u^{l+1}_n=u^{\perp}_j$,  
then $\dot{v}^{l+1}_{n,1}\ldots  \dot{v}^{l+1}_{n,d}$ will denote the elements adjacent to $u^{l+1}_n$ under $\mathfrak{M}^2_n$. 

There are two cases:

\begin{enumerate}[label={\Alph*)}]
\item $(u^{l+1}_n,v_{u^l_n})\in \mathfrak{M}^2_n$, i.e.  $v_{u^l_n}=\dot{v}^{l+1}_{n,k}$ for some $1\leq k \leq d$. 
In this case $u^{l+1}_n$ cannot be  $u^{\perp}_j$ for any $j$.\label{caseA}

\item $(u^{l+1}_n,v_{u^l_n})\notin \mathfrak{M}^2_n$, i.e. there exists some $u\in \Gamma^{(n)\star}(-u_n)$, such that $v_{u^l_n}=v_k$ for some $1\leq k \leq d$, where $v_1\ldots v_d$ denote the elements adjacent to $u$ under $\mathfrak{M}^2_n$. 
In this case it is possible that $u^{l+1}_n$ coincides with some $u^{\perp}_j$. \label{caseB}
\end{enumerate}

In either case we produce a cycle of length 2 by including the pair $(u^{l+1}_n,v_{u^{l}_n})$ into $M^{l+1}_n$. 
In fact we include it into $M^{l+1}_n$ together with a fan with the root $u^{l+1}_n$ and $(d-2)$ leaves taken among $\dot{v}^{l+1}_{n,i}$. 
To be precise we organize it as follows. 

In case A) if $k=d$, we set $v^{l+1}_{n,i-1}:=\dot{v}^{l+1}_{n,i}$ for $2\leq i\leq d$.
If $k\neq d$ we set $v^{l+1}_{n,i}:=\dot{v}^{l+1}_{n,i}$ for $1\leq i\leq d-1$. 
The set $M^{l+1}_n$ consists of edges $(u^{l+1}_n,v^{l+1}_{n,i})$ for $1\leq i\leq d-1$. 
The procedure is finished.

In case \ref{caseB}, $v^{l+1}_{n,1}:=v_{u^{l}_n}$ and $v^{l+1}_{n,i+1}:=\dot{v}^{l+1}_{n,i}$, $1 \le i \le d-2$. 
The set $M^{l+1}_n$ consists of edges $(u^{l+1}_n,v^{l+1}_{n,i})$ for $1\leq i\leq d-1$. 
We define $\dot{u}_n^{\perp}:=u$ and rename the remaining $d-1$ vertices $v_j$ to $\dot{v}^{\perp}_{n,i}$.
The procedure is finished. 

It is worth noting here that if $u^{l+1}_n$ coincides with some $u^{\perp}_j$, then the vertex $\dot{v}^{l+1}_{n,d}$ does not exist, i.e the only vertex $\dot{v}^{l+1}_{n,d-1}$ from the fan of $u^{l+1}_n$ would be outside of $M^{l+1}_n$. 
\vspace{0.5cm}
\end{quotation}

Figure \ref{fig:part3} illustrates part 3 of the algorithm. 

\begin{figure}[ht]  \centering
    \includegraphics[width=1\textwidth]{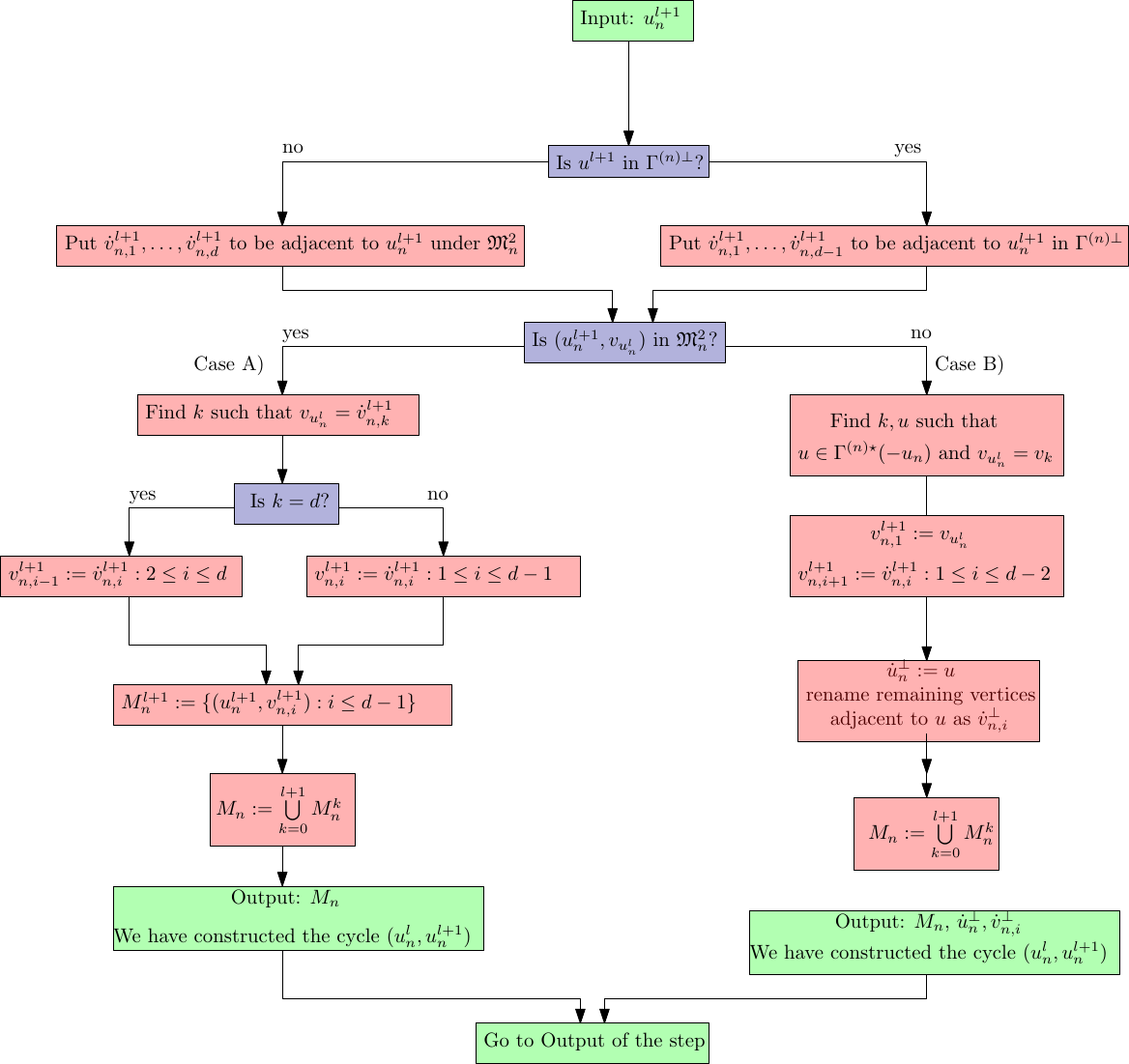}
    \caption{Detailed version of the algorithm used at the part 3 of the step $n+1$.}\label{fig:part3}

\end{figure}

\pagebreak 

Let $M_n=\bigcup\limits_{k=0}^{l+1} M^k_n$.
We obtain the graph $\Gamma^{(n+1)}$ from $\Gamma^{(n)}$ by removal of $M_n$-fans.
Since for each $u \in U^{(n)}\setminus U^{(n+1)}$ the element $v_u$ is also removed, then the graph is $U^{(n+1)}$-reflected. 

\subsection{Output of the step} \label{otpt} 
We have already defined $M_n$ and $\Gamma^{(n+1)}$. 
It remains to produce graphs $\Gamma^{(n+1)\star},\Gamma^{(n+1)\perp}$. 
For this purpose, we define auxiliary graphs 
$$ 
\mathfrak{T}=(\mathfrak{U},\mathfrak{V}) \mbox{ and }\dot{\Gamma}^{(n)\perp} = (\dot{U}^{(n)\perp},\dot{V}^{(n)\perp}).
$$

In cases (1), (2) and (3)A) of parts 2 and 3 we set $\mathfrak{U}:=U^{(n+1)}\setminus U^{(n)\perp}, \mathfrak{V}:=V^{(n+1)}\setminus V^{(n)\perp}$, and $\dot{\Gamma}^{(n)\perp}:=\Gamma^{(n)\perp}\cap \Gamma^{(n+1)}$.
Note here that in cases (1) and (2), in part 2 of the step, only elements of $\Gamma^{(n)\perp}$ were added to $M_n$. 
Therefore in these cases $\mathfrak{T}$ coincides with  $\Gamma^{(n)\star}(-u_n)$ and satisfies Hall's $d$-harem condition.

In case (3)B) we set 
\begin{itemize} 
\item $\mathfrak{U}:=U^{(n+1)}\setminus (U^{(n)\perp}\cup \{\dot{u}_n^{\perp}\})$, 
\item $\mathfrak{V}:=V^{(n+1)}\setminus (V^{(n)\perp} \cup \{\dot{v}_{n,i}^{\perp}: 1\leq i\leq d-1\})$,  
\end{itemize} 
unless $u^{l+1}_n$ coincides with some $u^{\perp}_j$. 
In the latter case, $\dot{v}^{l+1}_{n,d-1}$ is the only vertex left from the fan of $u^{\perp}_j$ in $\Gamma^{(n)\perp}$. 
\footnote{Furthermore, the vertex $\dot{v}^{l+1}_{n,d}$ does not exist, i.e the only vertex from the fan of $u^{l+1}_n$ that is outside of $M^{l+1}_n$ is $\dot{v}^{l+1}_{n,d-1}$.}
Consequently, we define 

$\bullet \, \mathfrak{V}:= (V^{(n+1)}\cup \{ \dot{v}^{l+1}_{n,d-1}\})\setminus (V^{(n)\perp} \cup \{\dot{v}_{n,i}^{\perp}: 1\leq i\leq d-1\})$.

We define $\dot{U}^{(n)\perp}$ and  $\dot{V}^{(n)\perp}$ accordingly: 
\begin{itemize} 
\item $\dot{U}^{(n)\perp}:=(U^{(n)\perp}\cup \{\dot{u}_n^{\perp}\})\cap U^{(n+1)}$, 
\item  $\dot{V}^{(n)\perp}:=(V^{(n)\perp}\cup \{\dot{v}_{n,k}^{\perp}:1\leq k \leq d-1\})\cap V^{(n+1)}$,  
\end{itemize} 

The following claim follows from Lemma \ref{2ndsim} below.
\begin{clm}\label{cs1p2}
At least one of the following holds:
\begin{itemize}
\item $\mathfrak{T}$ satisfies Hall's $d$-harem condition;
\item there exists some vertex $u_j^{\perp}\in \dot{U}^{(n)\perp}$ such that the graph $\mathfrak{T}(+u_j^{\perp})$ satisfies Hall's $d$-harem condition.
\item there exist vertices $u_i^{\perp},u_j^{\perp}\in \dot{U}^{(n)\perp}$ such that the graph $\mathfrak{T}(+u_i^{\perp},+u_j^{\perp})$ satisfies Hall's $d$-harem condition.
\end{itemize}
\end{clm} 
Now, depending on the output of the claim above we define graphs $\Gamma^{(n+1)\star},\Gamma^{(n+1)\perp}$.

In case (3)B) (i.e. $\dot{u}_n^{\perp}$ exists) if  $u_i^{\perp}\neq \dot{u}_n^{\perp}\neq u_j^{\perp}$ (for any output of the claim), we set $u_n^{\perp}:=\dot{u}_n^{\perp}$ and $v_{n,k}^{\perp}:= \dot{v}_{n,k}^{\perp}$ for $1\leq k\leq d-1$. 
Otherwise, or in the remaining cases $u_n^{\perp},v_{n,k}^{\perp}$ are not defined.

In the first case of the claim we set $\Gamma^{(n+1)\star}:=\mathfrak{T}$ and \begin{itemize} 
\item $U^{(n+1)\perp}:= (U^{(n)\perp}\cup\{u_n^{\perp}\})\cap U^{(n+1)},$
\item $V^{(n+1)\perp}:=((V^{(n)\perp}\setminus\{\dot{v}^{l+1}_{n,d-1}\})\cup \{v_{n,k}^{\perp}: 1\leq k\leq d-1\})\cap V^{(n+1)}.$
\end{itemize} 
In the second case we set $\Gamma^{(n+1)\star}:=\mathfrak{T}(+u_j^{\perp})$ and 
\begin{itemize} 
\item $U^{(n+1)\perp}:= ((U^{(n)\perp}\setminus\{u_j^{\perp}\})\cup \{u_n^{\perp}\})\cap U^{(n+1)},$
\item $V^{(n+1)\perp}:=(((V^{(n)\perp}\setminus\{\dot{v}^{l+1}_{n,d-1}\})\setminus \{v_{j,k}^{\perp}: 1\leq k\leq d-1\})\cup \{v_{n,k}^{\perp}: 1\leq k\leq d-1\})\cap V^{(n+1)}.$
\end{itemize} 
In the third case we set $\Gamma^{(n+1)\star}:=\mathfrak{T}(+u_i^{\perp},+u_j^{\perp})$ and 
\begin{itemize} 
\item $U^{(n+1)\perp}:= ((U^{(n)\perp}\setminus \{u_i^{\perp},u_j^{\perp}\})\cup\{u_n^{\perp}\})\cap U^{(n+1)},$
\item $V^{(n+1)\perp}:=(((V^{(n)\perp}\setminus\{\dot{v}^{l+1}_{n,d-1}\})\setminus \{v_{i,k}^{\perp},v_{j,k}^{\perp}: 1\leq k\leq d-1\})\cup \{v_{n,k}^{\perp}: 1\leq k\leq d-1\})\cap V^{(n+1)}.$ 
\end{itemize}

\section{General Lemmas}\label{lemmastechnic}

The construction presented in Section 3 is supported by the lemmas of Section 5. 
In order to prove them, we need some additional observations of slightly more general character. 
This is the purpose of this section. 
We warn the reader that the notation used here does not coincide with that of Section 3.
\begin{itemize} 
\item Throughout this section $\Gamma=(U,V,E)$ denotes a bipartite graph  \\ 
and $\Gamma^{\perp}=(U^{\perp},V^{\perp},E^{\perp})$ denotes its subgraph. \\ 
The graph $\Gamma^{\star}=(U^{\star},V^{\star},E^{\star})$ is an induced subgraph of $\Gamma$ such that 
$$
U^{\star}\cap U^{\perp}=\emptyset = V^{\star}\cap V^{\perp}.  
$$ 
\item Below we always assume that $d$ is a natural number greater than $1$. 
\end{itemize}  
The following situation will arise several times in our arguments in Section \ref{lemmassimplecase}. 
Let $X$ be a subset of $V$ such that  
$$|N_{\Gamma}(X)|\geq \frac{1}{d}|X|,$$
but
$$
|N_{\Gamma}(X) \setminus U^{\perp}|<\frac{1}{d}|X|.
$$
Thus we can conclude that $N_{\Gamma}(X)\cap U^{\perp}\neq \emptyset$.

The following lemma describes typical circumstances which lead to this situation.

\begin{lm}\label{neighsize}
Let $\Gamma=(U,V,E)$ be $U$-reflected and  
let $\Gamma^{\star}$ be a subgraph of $\Gamma$ induced by the sets of vertices $U^{\star}=U\setminus U^{\perp}$, $V^{\star}=V\setminus V^{\perp}$.
Assume that $\Gamma^{\star}$ satisfies Hall's $d$-harem condition, and assume that for each $Y\subset U^{\perp}$ we have 
$|N_{\Gamma}(Y) \cap V^{\perp}|\geq (d-1)|Y|$.

Then for any $X\subseteq V$ we have $|N_{\Gamma}(X)|\geq (d-1)|X|$.
In particular $|N_{\Gamma}(X)|\geq \frac{1}{d}|X|$
\end{lm}
\begin{proof}
Let $U_X:=\{u\in U: v_u\in X\}$.
Since $\Gamma$ is $U$-reflected, $|U_X|=|X|$.
Consider the sets 
$$U_X^{\star}:=U_X\setminus U^{\perp}$$
and 
$$U_X^{\perp}:=U_X\cap U^{\perp}.$$  
Using $U$-reflectedness of $\Gamma$ again, we see  
$$
|N_{\Gamma}(X)|\geq |N_{\Gamma^{\star}}(U_X^{\star})|+|N_{\Gamma}(U^{\perp}_X)\cap V^{\perp}|.$$
Since Hall's $d$-harem condition is satisfied for any subset of $U^{\star}$, 
$$
|N_{\Gamma^{\star}}(U_X^{\star})|\geq d|U_X^{\star}|.
$$
Moreover we have $|N_{\Gamma}(U_X^{\perp})\cap V^{\perp}|\geq (d-1)|U_X^{\perp}|$.
Therefore
$$|N_{\Gamma^{\star}}(U_X^{\star})|+|N_{\Gamma}(U^{\perp}_X)\cap V^{\perp}|\geq (d-1)|U_X|.$$
Since $d\geq 2$, it follows that 
$$|N_{\Gamma}(X)|\geq (d-1)|X|\geq \frac{1}{d}|X|.$$

\end{proof}

The conditions of this lemma do not fit to the situation of Part 1 of the step of the main construction. 
Indeed, when some vertices from the graph are removed, we cannot guarantee that it is still $U$-reflected. 
On the other hand, it can happen that for some $v\in V$, the subgraph $\Gamma(-v)$ is $U$-reflected. 
This is the reason why in the following lemmas both $\Gamma^{\perp}$ and $\Gamma^{\star}$ are  subgraphs of some $\Gamma(-v)$.

\begin{lm}\label{neighsize2}
Let $v$ be a vertex from $V$ and 
$\Gamma(-v)=(U,V(-v))$. 
Assume that $\Gamma(-v)$ is $U$-reflected.
Let $\Gamma^{\perp}=(U^{\perp},V^{\perp},E^{\perp})$ be a subgraph of $\Gamma(-v)$ and let $\Gamma^{\star}$ be defined as the subgraph of $\Gamma(-v)$ induced by the sets of vertices 
$U^{\star}:=U\setminus U^{\perp},V^{\star}:=V(-v)\setminus V^{\perp}$.

Assume that $\Gamma^{\star}$ satisfies Hall's $d$-harem condition, and assume that for each $Y\subset U^{\perp}$ we have 
$|N_{\Gamma}(Y)\cap V^{\perp}|\geq (d-1)|Y|$.

Then for any $X \subseteq V\setminus\{v\}$ the inequality 
$|N_{\Gamma}(X)|\geq (d-1)|X|-1$ holds. \\ 
Moreover, the equality $|N_{\Gamma}(X)|=(d-1)|X|-1$ can happen only if $v\in N_{\Gamma}(U_X)$ and 
$N_{\Gamma^{\star}}(U_X)=\emptyset$, 
where $U_X =\{u\in U: v_u\in X\}$.  

\end{lm}
\begin{proof}
By $U$-reflectedness of $\Gamma(-v)$ we have $|U_X|=|X|$ and 
$$
|N_{\Gamma}(X)|+1\geq|N_{\Gamma(-v)}(U_X)|.
$$
Furthermore, if $v\notin N_{\Gamma}(U_X)$, then 
$$
|N_{\Gamma}(X)|\geq|N_{\Gamma(-v)}(U_X)|.
$$ 
Consider the sets 
$$
U_X^{\star}:=U_X\setminus U^{\perp} \, \mbox{ and } \, 
U_X^{\perp}:=U_X\cap U^{\perp}.
$$
Applying the argument of Lemma \ref{neighsize} to $\Gamma (-v)$ we obtain 
$$
|N_{\Gamma(-v)}(U_X)| \geq |N_{\Gamma^{\star}}(U_X^{\star})|+|N_{\Gamma}(U^{\perp}_X)\cap V^{\perp}| \geq (d-1)|U_X|.
$$
Thus 
$|N_{\Gamma}(X)|\geq |N_{\Gamma(-v)}(U_X)|-1\geq (d-1)|X|-1$.

Observe that $|N_{\Gamma(-v)}(U_X)|=(d-1)|U_X|$ only if $N_{\Gamma(-v)}(U_X)\subseteq V^{\perp}$, i.e. $N_{\Gamma^{\star}}(U_X)=\emptyset$ (since Hall's $d$-harem condition is satisfied for any subset of $U^{\star}$). 
If $v\notin N_{\Gamma}(U_X)$, then 
$$|N_{\Gamma}(X)|\geq |N_{\Gamma(-v)}(U_X)| \geq (d-1)|X|.$$
Therefore $|N_{\Gamma}(X)|=(d-1)|X|-1$ only if $v\in N_{\Gamma}(U_X)$ and $N_{\Gamma^{\star}}(U_X)=\emptyset$.

\end{proof}
The following definition will be useful in the proofs of the lemmas of the next section.
\begin{df}\label{acc}
Let $\Gamma:=(U,V,E)$ be a bipartite graph.  Assume that 
$\Gamma^{\star}=(U^{\star},V^{\star},E^{\star})$ is a subgraph of $\Gamma$ satisfying Hall's $d$-harem condition, and $M$ is a perfect $(1,d)$-matching in $\Gamma^{\star}$.

Let $X$ be a subset of $V$ and let $x\in X$.
We say that $x$ \textit{ is accessible from } $y\in N_{\Gamma}(X)$ \textit{ through } $X$ \textit{ by matching } $M$, (denoted by 
$y \xhookdoubleheadrightarrow{M,X} x$) if there exist two sequences of vertices $\{v'_0,\ldots,v'_n\}\subseteq X$ and 
$\{u'_0,\ldots,u'_{n-1}\}\subseteq N_{\Gamma}(X)$ such that 
\begin{itemize}
\item $v'_n=x$;
\item $(u'_i,v'_{i})\in M$ and $(u'_{i},v'_{i+1})\in E\setminus M$, where $i <n$;
\item $(y,v'_0)\in E$.
\end{itemize}
\end{df}

\begin{figure}[h]
  \centering
    \includegraphics[width=0.7\textwidth]{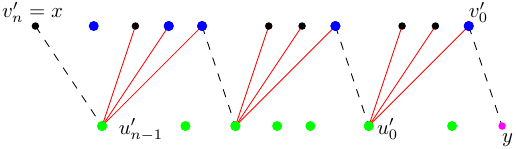}
	\caption{$x$ is accessible from $y$ by matching $M$ (red), the elements of $X$ are  blue and the elements of $N_{\Gamma}(X)$ are green}
\end{figure}

\begin{lm}\label{access}
Let $\Gamma=(U,V,E)$ be a bipartite graph. 
Assume that $\Gamma^{\star}=(U^{\star},V^{\star},E^{\star})$ is a subgraph of $\Gamma$ satisfying Hall's $d$-harem condition, and $M$ is a perfect $(1,d)$-matching in $\Gamma^{\star}$.

Assume that 
\begin{itemize}
\item $\widehat{v}\in N_{\Gamma}(U^{\star})\setminus V^{\star}$;
\item $X$ is a minimal connected subset in $V^{\star}\cup \{\widehat{v}\}$ with $|N_{\Gamma}(X)|< \frac{1}{d}|X|$;
\item $\widehat{u}\in N_{\Gamma}(X)\setminus U^{\star}$. 
\item $(\widehat{u},\widehat{v})\notin E$. 
\end{itemize}
Then $\widehat{u} \xhookdoubleheadrightarrow{M,X} \widehat{v}$.
\end{lm}

\begin{proof}
It is clear that $\widehat{v}\in X$. 
Let $X'$ denote the subset of all elements of $X$ that are either adjacent to $\widehat{u}$ or accessible from $\widehat{u}$ through $X$ by the matching $M$.  
Since $\widehat{u}\in N_{\Gamma}(X)$, $X'\neq \emptyset$. 
In order to get a contradiction, assume that $\widehat{v}\not\in X'$. 
We will show that 
$|N_{\Gamma}(X\setminus X')|< \frac{1}{d}|X\setminus X'|$. 

Let $|X'|=l$ and let 
$U_M (X')=\{ u\in U \, | \, (\exists v\in X') (u,v)\in M\}$. 
Since elements of $X\setminus X'$ are not accessible from $\widehat{u}$ through $X$ by the matching $M$, then 
$N_{\Gamma}(X \setminus X') \cap U_M (X') = \emptyset$. 

Using this we see that 
$$
|N_{\Gamma}(X\setminus X')|\leq |N_{\Gamma}(X)|-|U_M(X')|. 
$$
Since $M$ is a $(1,d)$-matching, each element of $U_M(X')$ can be matched with at most $d$ elements from $X'$. 
Therefore $|U_M(X')|\geq \lceil\frac{l}{d}\rceil$ and 
$$
|N_{\Gamma}(X\setminus X')|\leq |N_{\Gamma}(X)|-|U_M(X')|\leq |N_{\Gamma}(X) |-\lceil\frac{l}{d}\rceil<\frac{1}{d}(|X|-l)=\frac{1}{d}|X\setminus X'|. 
$$ 
So $|N_{\Gamma}(X\setminus X')|< \frac{1}{d}|X\setminus X'|$  and $X\setminus X'$ is smaller than $X$, a contradiction with the choice of the latter.
As a result $\widehat{v}\in X'$, and consequently $\widehat{u} \xhookdoubleheadrightarrow{M,X} \widehat{v}$.
\end{proof}

\section{Graphs constructed in parts 1 and 2 of each step satisfy Hall's $d$-harem condition}\label{lemmassimplecase}

In this section the notation is taken from the construction of Section 3.
\subsection{Claims  \ref{cs1p1}, \ref{csnp1}}
Before stating Lemma \ref{1stsim} (proving Claims \ref{cs1p1} and \ref{csnp1}) we recall some notation used in case (2) of the $n+1$-st step of the construction.

\begin{itemize}
\item $\Gamma^{(n)}$ is $U^{(n)}$-reflected;
\item $\Gamma^{(n)\star}$ is a subgraph of $\Gamma^{(n)}$ obtained by removal of $(1,d-1)$-fans with roots belonging to the set $U^{(n)\perp}$; 
\item $\Gamma^{(n)\star}$ satisfies $d$-harem condition, and $u_n \in \Gamma^{(n)\star}$;
\item $|U^{(n)\perp}|\leq n$;
\item $\mathfrak{M}^1_n$ is a $(1,d)$-matching in the graph $\Gamma^{(n)\star}$. 
\end{itemize}

\begin{lm}\label{1stsim}
For any $n$ one of the following statements holds:
\begin{itemize}
\item $\Gamma^{(n)\star}(-u_n)$ satisfies Hall's $d$-harem condition;
\item there exist some vertex $u_j^{\perp}\in U^{(n)\perp}$ such that the graph $\Gamma^{(n)\star}(-u_n,+u_j^{\perp})$ satisfies Hall's $d$-harem condition.
\end{itemize}
\end{lm}

\begin{rem}
If $U^{(n)\perp}=\emptyset$, then by Lemma \ref{1stsim} the graph $\Gamma^{(n)\star}(-u_n)$ satisfies Hall's $d$-harem condition. 
In particular, Claim \ref{cs1p1} holds.
\end{rem}

\begin{proof}
We know that the graph $\Gamma^{(n)\star}$ satisfies Hall's $d$-harem condition. 
Let $\mathfrak{v}$ denote the only vertex from the set $\{v^0_{n,1},\ldots ,v^0_{n,d}\}$ that belongs to $\Gamma^{(n)\star}(-u_n)$.
The choice of $u_n, v^0_{n,1},\ldots ,v^0_{n,d}$ ensures that the graph $\Gamma^{(n)\star}(-u_n,-\mathfrak{v})$ satisfies Hall's $d$-harem condition as well. 
For further arguments: 

$\bullet$ let $\mathfrak{M}$ denote a perfect $(1,d)$-matching in $\Gamma^{(n)\star}(-u_n,-\mathfrak{v})$. 

Since $U^{(n)\star}(-u_n)=U^{(n)\star}(-u_n,-\mathfrak{v})$, for $X\subset U^{(n)\star}(-u_n)$ we have 
\begin{equation}\label{Uhcond}
|N_{\Gamma^{(n)\star}(-u_n)}(X)|\geq |N_{\Gamma^{(n)\star}(-u_n,-\mathfrak{v})}(X)|\geq d|X|.
\end{equation}

The corresponding property also holds for all subsets of $V^{(n)\star}(-u_n)$ that do not contain $\mathfrak{v}$. 
Therefore if $\Gamma^{(n)\star}(-u_n)$ does not satisfy Hall's $d$-harem condition, then a witness of this is a finite 
$X \subset V^{(n)\star}(-u_n)$ which contains $\mathfrak{v}$.

If such $X$ is not connected, then the neighbourhood of $X$ is a disjoint union of the neighbourhoods of the connected subsets of $X$. 
Therefore there exists a minimal connected set $X$ such that $\mathfrak{v}\in X\subset V^{(n)\star}(-u_n)$ and $|N_{\Gamma^{(n)\star}(-u_n)}(X)|<\frac{1}{d}|X|$.
Choose $X$ with these properties. 
We will now prove the lemma into two steps.  \\ 
$\bullet$ {\em Step 1. There exists some} $u_j^{\perp}\in N_{\Gamma^{(n)}(-u_n)}(X)$. \\  
This is the main part of the proof of the lemma. 
In fact we will prove that 
\begin{equation}\label{0}
|N_{\Gamma^{(n)}(-u_n)}(X)|\geq \frac{1}{d}|X|.
\end{equation} 
Then the existence of the required $u^{\perp}_j$ follows from the inequality $|N_{\Gamma^{(n)\star}(-u_n)}(X)|<\frac{1}{d}|X|$. 

First, consider the case when 
$X=\{\mathfrak{v}\}$. 
The vertex $u_{\mathfrak{v}}$ either belongs to $U^{(n)\star}(-u_n)$ or to $U^{(n)\perp}$. 
By inequality (\ref{Uhcond}) and the fact that $U^{(n)\perp}$ consist of $(1,d-1)$-fans, in each case we have $|N_{\Gamma^{(n)}(-u_n)}(u_{\mathfrak{v}})|\geq d-1$.
Moreover, if the equality
\begin{equation}\label{11}
|N_{\Gamma^{(n)}(-u_n)}(u_{\mathfrak{v}})|= d-1
\end{equation} 
holds, then $v_{u_n}\not\in \Gamma^{(n)}$. 
Indeed, if $v_{u_n} \in \Gamma^{(n)}$, then $v_{u_n}\in N_{\Gamma^{(n)}(-u_n)}(u_{\mathfrak{v}})$ (by reflectedness).
On the other hand, the existence of $\mathfrak{M}$ together with equality (\ref{11}) implies that $(u_{\mathfrak{v}},v_{u_n})\in\Gamma^{(n)\perp}$. 
However, the first part of the $n+1$-st step of the construction implies that $(u_n,\mathfrak{v})$ should be added to the matching $M_n$ already: the element $\mathfrak{v}$ plays the role of $v^0_{n,j}$ in case (2) of the first part. 
In particular, $\mathfrak{v}\notin \Gamma^{(n)\star}(-u_n)$, a contradiction.

Our next observation is 
$$
|N_{\Gamma^{(n)}(-u_n)}(\mathfrak{v})| \geq |N_{\Gamma^{(n)}(-u_n)}(u_\mathfrak{v})|-1. 
$$
This follows by $U^{(n)}$-reflectedness of $\Gamma^{(n)}$: $v_{u_n}$ is the only possible element adjacent to $u_{\mathfrak{v}}$ that does not have the left copy in $U^{(n)}(-u_n)$.
Additionally, note that the equality 
\begin{equation}\label{22}
|N_{\Gamma^{(n)}(-u_n)}(\mathfrak{v})|= |N_{\Gamma^{(n)}(-u_n)}(u_\mathfrak{v})|-1
\end{equation}
holds only if $v_{u_n}\in \Gamma^{(n)}(-u_n)$.
 
Therefore equalities (\ref{11}), (\ref{22}) are not consistent, i.e.:
$$
|N_{\Gamma^{(n)}(-u_n)}(\mathfrak{v})|	 >  (d-1)-1 \mbox{ and, in particular, } 
|N_{\Gamma^{(n)}(-u_n)}(\mathfrak{v})| \geq \frac{1}{d}.
$$ 
Since $X$ is a singleton,
$$
|N_{\Gamma^{(n)}(-u_n)}(X)|\geq \frac{1}{d}|X|.
$$
As a result we have inequality (2) and the fact  that there is some 
$u_j^{\perp}\in N_{\Gamma^{(n)}(-u_n)}(X)$.  

Consider this case when $X\not=\{ \mathfrak{v}\}$. Then $|N_{\Gamma^{(n)\star}(-u_n)}(X)|\geq 1$ and 
by the assumption 
$|N_{\Gamma^{(n)\star}(-u_n)}(X)|<\frac{1}{d}|X|$, we have $|X|>d$. 

Let $U_X:=\{u\in U: v_u\in X\}$ in $\Gamma^{(n)}(-u_n)$. 
Applying the fact that $V^{(n)}$ is a subset of the right copy of $U^{(n)}$ we arrive at two possibilities:
\begin{enumerate}[(i)]
\item $v_{u_n}\notin X$ and  $|X|=|U_X|$;
\item $v_{u_n}\in X$ and $|X|=|U_X|+1$.
\end{enumerate} 
We claim that in either case the inequality (\ref{0}) follows from Lemma \ref{neighsize2}.

\begin{figure}[h]
  \centering
    \includegraphics[width=0.7\textwidth]{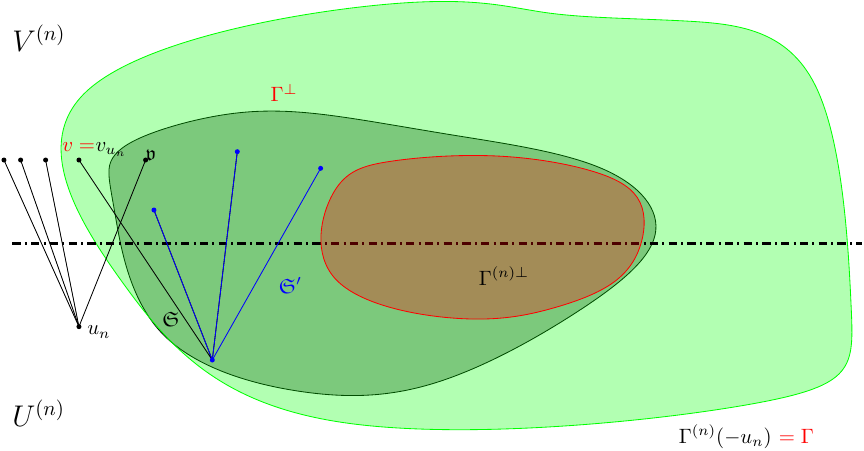}
	\caption{Application of Lemma \ref{neighsize2}; red letters correspond to the  notation of Lemma \ref{neighsize2}.}
\end{figure}

Indeed, let $\mathfrak{S}$ denote the fan from the matching $\mathfrak{M}$ containing $v_{u_n}$ and $\mathfrak{S}'$ denote $\mathfrak{S}$ with $v_{u_n}$ removed.
The conditions of Lemma \ref{neighsize2} are satisfied if we consider $\Gamma^{(n)}(-u_n)$ as $\Gamma$ in that lemma, $v_{u_n}$ as $v$, and 
$\Gamma^{(n)\perp}\cup\{\mathfrak{v}\}\cup \mathfrak{S}'$ as $\Gamma^{\perp}$. 
Indeed, $\Gamma^{(n)}(-u_n,-v_{u_n})$ is $U^{(n)}(-u_n,-v_{u_n})$-reflected. 
Moreover, the corresponding graph $\Gamma^{\star}$ from Lemma \ref{neighsize2} is obtained by removal of $\Gamma^{(n)\perp}\cup\{\mathfrak{v}\}\cup\mathfrak{S}'$ from $\Gamma^{(n)}(-u_n,-v_{u_n})$, therefore it is the same as $\Gamma^{(n)\star}(-u_n,-\mathfrak{v})$ with $\mathfrak{S}$ removed. 
Since $\mathfrak{S}$ is a fan from a perfect $(1,d)$-matching in $\Gamma^{(n)\star}(-u_n,-\mathfrak{v})$, we know that $\Gamma^{\star}$ satisfies Hall's $d$-harem condition.

Therefore in case (i), by Lemma \ref{neighsize2}:
$$|N_{\Gamma^{(n)}(-u_n)}(X)|\geq (d-1)|X|-1.$$
This fact combined with inequalities $d\geq 2$ and $|X|>d$ implies that 
$$|N_{\Gamma^{(n)}(-u_n)}(X)|\geq\frac{1}{d}|X|.$$
In case (ii), by Lemma \ref{neighsize2}:
$$|N_{\Gamma^{(n)}(-u_n)}(X\setminus \{v_{u_n}\})|\geq (d-1)(|X|-1)-1, 
$$ 
$$
\mbox{ i.e. } \, |N_{\Gamma^{(n)}(-u_n)}(X)| \geq (d-1)(|X|-1)-1.
$$
Note that the inequality is strict. 
Indeed, by Lemma \ref{neighsize2} the equality 
$$
|N_{\Gamma^{(n)}(-u_n)}(X\setminus \{v_{u_n}\})|=(d-1)(|X|-1)-1
$$
implies that 
$v_{u_n}\in N_{\Gamma^{(n)}(-u_n)}(U_{X\setminus \{v_{u_n}\}})$ and $N_{\Gamma^{(n)\star}(-u_n,-\mathfrak{v})}(U_{X\setminus \{v_{u_n}\}})=\emptyset$.
Thus, $v_{u_n}\in V^{(n)\perp}$. 
On the other hand, $v_{u_n}\in X\subset \Gamma^{(n)\star}$, a contradiction with the choice of $X$.

Again, inequalities $d\geq 2$ and $|X|>d$ imply 
$$|N_{\Gamma^{(n)}(-u_n)}(X)|\geq(d-1)(|X|-1)\geq\frac{1}{d}|X|.$$ 
This finishes our argument for the statement that the assumption 
$|N_{\Gamma^{(n)\star}(-u_n)}(X)|<\frac{1}{d}|X|$ implies 
$$
N_{\Gamma^{(n)}(-u_n)}(X)\cap U^{(n)\perp}(-u_n)\neq\emptyset , 
$$
i.e. there exists some $u_j^{\perp}\in N_{\Gamma^{(n)}(-u_n)}(X)$. 
Let us denote by $v^{\perp}_{j,1},\ldots, v^{\perp}_{j,d-1}$ 
the $(d-1)$-tuple of vertices adjacent to $u_j^{\perp}$ in $\Gamma^{(n)\perp}$. \\ 
$\bullet$ {\em Step 2. $\Gamma^{(n)\star}(-u_n,+u_j^{\perp})$ satisfies Hall's $d$-harem condition.} \\ 
We start with an application of Lemma \ref{access}.
Consider the induced subgraph of the graph $\Gamma^{(n)\star}(-u_n,-\mathfrak{v})$ that consists of all vertices adjacent to the edges of the matching $\mathfrak{M}^1_n$. 
We take it as the graph $\Gamma^{\star}$ from Lemma \ref{access}. 
Since $\mathfrak{M}^1_n$ is a perfect $(1,d)$-matching in the graph $\Gamma^{(n)\star}$, it follows that $\Gamma^{\star}$ satisfies Hall's $d$-harem condition. 
The matching $\mathfrak{M}^1_n$ plays the role of $M$ from that lemma and the bipartite graph $\Gamma^{(n)}(-u_n)$ plays the role of $\Gamma$. 
Apply Lemma \ref{access} for $\mathfrak{v}$, $X$, $u^{\perp}_j$.   
We see $u^{\perp}_j \xhookdoubleheadrightarrow{\mathfrak{M}^1_n,X} \mathfrak{v}$. 
This gives us sequences of vertices $\{v_0',\ldots,v'_n\}, \{u_0',\ldots,u'_{n-1}\}$ as in Definition \ref{acc}.

In order to prove that the graph $\Gamma^{(n)\star}(-u_n,+u_j^{\perp})$ satisfies Hall's $d$-harem condition, we construct a perfect $(1,d)$-matching in it. 
We set
$$
M':=(\mathfrak{M}^1_n\setminus\{(u_n,v^0_{n,1}),\ldots, (u_n,v^0_{n,d}),(u'_{0},v'_{0}),\ldots (u'_{n-1},v'_{n-1})\}) \cup
$$ 
$$
\{(u^{\perp}_j,v^{\perp}_{j,1}),\ldots,(u^{\perp}_j,v^{\perp}_{j,d-1}),(u^{\perp}_j,v'_{0}), (u'_{0},v'_{1}),\ldots (u'_{n-1},v'_{n})\},
$$ 
where $v'_n=\mathfrak{v} \in \{ v^0_{n,1},\ldots, v^0_{n,d}\}$.

\begin{figure}[h]
  \centering
    \includegraphics[width=0.7\textwidth]{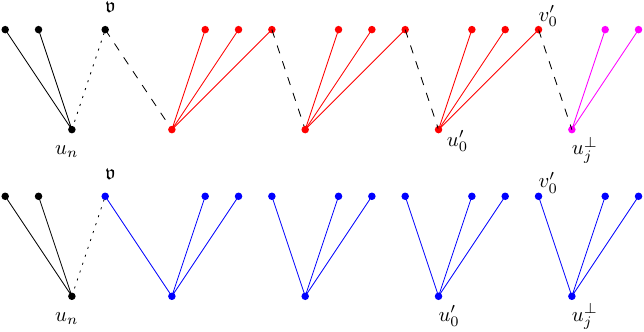}
	\caption{We replace the red fans in the matching $\mathfrak{M}^1_n$ by the blue fans to obtain the matching $M'$ in $\Gamma^{(n)\star}(-u_n,+u_j^{\perp})$.}
\end{figure}

We remind the reader that $\mathfrak{M}^1_n$ is a perfect $(1,d)$-matching in the graph $\Gamma^{(n)\star}$. 
We have obtained $M'$ by removing $d$ edges adjacent to $u_n$, adding $d$ edges incident to $u^{\perp}_j$, and the following replacement: for each of $u'_i$ we replace one edge incident to it by another incident edge (then $\mathfrak{v}$ becomes adjacent to one edge in $M'$).
It follows that the matching $M'$ is a perfect $(1,d)$-matching in the graph $\Gamma^{(n)\star}(-u_n,+u_j^{\perp})$. Therefore that graph satisfies Hall's $d$-harem condition.
\end{proof}

\subsection{Notation used in proof of Lemma \ref{2ndsim}}
Before stating the second lemma, we remind the reader the notation used in it.
\begin{itemize}
\item In Case 3A
\begin{itemize}
\item $\mathfrak{U}:=U^{(n+1)}\setminus U^{(n)\perp}$
\item $\mathfrak{V}:=V^{(n+1)}\setminus V^{(n)\perp}$
\item $\mathfrak{T}:=(\mathfrak{U},\mathfrak{V},\mathfrak{E})$, where $\mathfrak{E}$ is induced in $\Gamma$ by the sets of vertices $\mathfrak{U},\mathfrak{V}$. 
\item $\dot{U}^{(n)\perp}=U^{(n)\perp}\cap U^{(n+1)}$;\item $\dot{V}^{(n)\perp}=V^{(n)\perp}\cap V^{(n+1)}$
\item $\dot{\Gamma}^{(n)\perp}:=(\dot{U}^{(n)\perp},\dot{V}^{(n)\perp},\dot{E}^{(n)\perp})$, where $\dot{E}^{(n)\perp}$ is defined according to the set of fans of elements from $\dot{U}^{(n)\perp}$ putted into $\dot{\Gamma}^{(n)\perp}$.
\end{itemize}
\item In Case 3B
\begin{itemize}
\item $\mathfrak{U}:=U^{(n+1)}\setminus (U^{(n)\perp}\cup \{\dot{u}_n^{\perp}\})$ 
\item $\mathfrak{V}:=V^{(n+1)}\setminus (V^{(n)\perp}\cup \{\dot{v}_{n,i}^{\perp}: 1\leq i\leq d-1\})$.
\item $\mathfrak{T}:=(\mathfrak{U},\mathfrak{V},\mathfrak{E})$, where $\mathfrak{E}$ is induced in $\Gamma$ by the sets of vertices $\mathfrak{U},\mathfrak{V}$.\item $\dot{U}^{(n)\perp}:= (U^{(n)\perp}\cup\{\dot{u}_n^{\perp}\})\cap U^{(n+1)}$ with $\dot{V}^{(n)\perp}$ defined accordingly: $\dot{V}^{(n)\perp}:=(V^{(n)\perp}\cup \{\dot{v}_{n,k}^{\perp}:1\leq k \leq d-1\})\cap V^{(n+1)}$; 
\item $\dot{\Gamma}^{(n)\perp}:=(\dot{U}^{(n)\perp},\dot{V}^{(n)\perp},\dot{E}^{(n)\perp})$, where $\dot{E}^{(n)\perp}$ is defined according to the set of fans of elements from $\dot{U}^{(n)\perp}$ putted into $\dot{\Gamma}^{(n)\perp}$.
\end{itemize}
\item $\mathfrak{M}^2_n$ is a perfect $(1,d)$-matching in the graph $\Gamma^{(n)\star}(-u_n)$; 
\item $\dot{v}^{l+1}_{n,i}$, $1 \le i \le d$ (or $d-1$), are vertices adjacent to $u^{l+1}_n$ in $\mathfrak{M}^2_n$ (or in $\Gamma^{(n)\perp}$);
\item $\dot{v}_1,\dot{v}_2$ are the elements among $\dot{v}^{l+1}_{n,i}$, $1 \le i \le d$ (or $d-1$), which are not added to $M^{l+1}_n$; there are at most two of them. 

\end{itemize}

\subsection{Claims  \ref{cs1p2part1}, \ref{cs1p2}}

\begin{lm}\label{2ndsim}
For any $n$ one of the following holds:
\begin{itemize}
\item $\mathfrak{T}$ satisfies Hall's $d$-harem condition;
\item there exist some vertex $u_j^{\perp}\in \dot{U}^{(n)\perp}$ such that the graph $\mathfrak{T}(+u_j^{\perp})$ satisfies Hall's $d$-harem condition.
\item there exists vertices $u_i^{\perp},u_j^{\perp}\in \dot{U}^{(n)\perp}$ such that the graph $\mathfrak{T}(+u_i^{\perp},+u_j^{\perp})$ satisfies Hall's $d$-harem condition.
\end{itemize}
\end{lm}

\begin{rem}
If $|\dot{U}^{(n)\perp}|\leq 1$, then Lemma \ref{2ndsim} can be restated as follows. 
One of	the following holds:
\begin{itemize}
\item $\mathfrak{T}$ satisfies Hall's $d$-harem condition;
\item $\Gamma^{(n+1)}$ satisfies Hall's $d$-harem condition.
\end{itemize}
Therefore, this lemma proves Claim \ref{cs1p2part1}.
\end{rem}

\begin{proof}[Proof of Lemma \ref{2ndsim}]
Assume that $\mathfrak{T}$ does not satisfy Hall's $d$-harem condition. 
Let $u^{l+1}_n$ be the root of the last fan added to the matching $M_n$ in the second part of the $n$-th step; it belongs to the produced cycle of length $2$.
Recall that $\dot{v}_1,\dot{v}_2$ denote the vertices of the form $\dot{v}^{l+1}_{n,i}$ that were not added to $M^{l+1}_n$.
In particular in Case 3A there is only one such a vertex and in Case 3B there is either one or two such vertices.

In the rest of the proof we will use the following two claims.

\begin{clm}\label{gn0}
Depending on the existence of $\dot{v}_2$ either the graph $\mathfrak{T}(-\dot{v}_1)$ or $\mathfrak{T}(-\dot{v}_1,-\dot{v}_2)$ satisfies Hall's $d$-harem condition.
\end{clm}

\begin{proof}
We only consider the case when $\dot{v}_2$ exists, the other one is analogous. 

It follows from the procedure that:
$$
\mathfrak{U}(-\dot{v}_1,-\dot{v}_2)=U^{(n)\star}(-u_n)\setminus\{u^{l+1}_n,\dot{u}^{\perp}_n\} 
$$
and 
$$
\mathfrak{V}(-\dot{v}_1,-\dot{v}_2)=V^{(n)\star}(-u_n)\setminus\{\dot{v}^{l+1}_{n,1},\ldots, \dot{v}^{l+1}_{n,d},\dot{v}^{\perp}_{n,1},\ldots, \dot{v}^{\perp}_{n,d}\}.
$$ 
Since $\mathfrak{M}^2_n$ is a perfect $(1,d)$-matching in $\Gamma^{(n)\star}(-u_n)$ and $\mathfrak{T}(-\dot{v}_1,-\dot{v}_2)$ is obtained from $\Gamma^{(n)\star}(-u_n)$ by removal of two fans from $\mathfrak{M}^2_n$, we know that $\mathfrak{T}(-\dot{v}_1,-\dot{v}_2)$ satisfies Hall's $d$-harem condition. 
\end{proof}

\begin{clm}\label{gn1}
For any $X\subset V^{(n+1)}$, we have
$$|N_{\Gamma^{(n+1)}}(X)|\geq\frac{1}{d}|X|.$$
\end{clm}
\begin{proof}
The inequality follows from Lemma \ref{neighsize}.
Indeed, consider $\Gamma^{(n+1)}$ to be $\Gamma$ from that lemma. 
Note that the construction guarantees that 
$\Gamma^{(n+1)}$ is $U^{(n+1)}$-reflected. 
Depending on existence of $\dot{v}_2$ consider either $\dot{\Gamma}^{(n)\perp}(+\dot{v}_1)$ or $\dot{\Gamma}^{(n)\perp}(+\dot{v}_1,+\dot{v}_2)$ to be $\Gamma^{\perp}$ from that lemma. 
Then the corresponding graph $\Gamma^{\star}$ from the lemma is equal to either $\mathfrak{T}(-\dot{v}_1)$ or $\mathfrak{T}(-\dot{v}_1,-\dot{v}_2)$ and by Claim \ref{gn0} it satisfies Hall's $d$-harem condition. 
Therefore, the conditions of Lemma \ref{neighsize} are satisfied.
\end{proof}

We now return to the main course of the proof. From now on we consider the case of two additional vertices: $\dot{v}_1$ and $\dot{v}_2$. 
The case of only one of them, $\dot{v}_1$, is similar and simpler (in particular, we do not need two parts in the further proof). 

It is clear that the inequality $|N_{\mathfrak{T}}(X)|\geq |N_{\mathfrak{T}(-\dot{v}_1,-\dot{v}_2)}(X)|$ holds for all $X\subset \mathfrak{U}$.
This inequality also holds for subsets of $\mathfrak{V}$ which do not intersect $\{ \dot{v}_1,\dot{v}_2\}$. 
On the other hand, applying Claim \ref{gn0} we see that $\mathfrak{T}(-\dot{v}_1,-\dot{v}_2)$ satisfies Hall's $d$-harem condition.  
Therefore, if $\mathfrak{T}$ does not satisfy Hall's $d$-harem condition, then it would be  witnessed by some finite subset of $\mathfrak{V}$ containing at least one of $\dot{v}_1,\dot{v}_2$.
 
The rest of the proof is divided into two parts.

\noindent 
$\bullet$ \textit{Part 1.} In this part of the proof we define a graph $\Gamma'$, a matching $M'$ and a family of fans $\dot{\Gamma}^{(n+1)\perp}$.
We start with checking whether the graph $\mathfrak{T}(-\dot{v}_2)$ satisfies Hall's $d$-harem condition. 
If it does, we set \begin{itemize}
\item $M'$ is a perfect $(1,d)$-matching in $\mathfrak{T}(-\dot{v}_2)$;
\item $\dot{\Gamma}^{(n+1)\perp}:=\dot{\Gamma}^{(n)\perp}$;
\item $\Gamma'=\mathfrak{T}$, denoting $\Gamma'=(U',V')$,  
\end{itemize} 
and finish Part 1 of the proof.

If it does not, then there exists a minimal connected set $X$ such that $\dot{v}_1\in X\subset \mathfrak{V}(-\dot{v}_2)$ and $|N_{\mathfrak{T}(-\dot{v}_2)}(X)|<\frac{1}{d}|X|$. 
Observe that $|N_{\Gamma^{(n+1)}}(X)|\geq\frac{1}{d}|X|$ by Claim \ref{gn1}.

The inequalities 
$$
|N_{\Gamma^{(n+1)}}(X)|\geq\frac{1}{d}|X| \, \mbox{ and } 
\, |N_{\mathfrak{T}(-\dot{v}_2)}(X)|<\frac{1}{d}|X|
$$
imply 
$$ 
N_{\Gamma^{(n+1)}}(X)\cap \dot{U}^{(n)\perp}\neq\emptyset,
$$
i.e. there exists some $u^{\perp}_j\in N_{\Gamma^{(n+1)}}(X)$. Similarly as in Lemma \ref{1stsim} we denote by $v^{\perp}_{j,1},\ldots, v^{\perp}_{j,d-1}$ the remaining vertices of the fan from $\dot{\Gamma}^{(n)\perp}$ containing $u^{\perp}_j$.

Since $\mathfrak{M}^2_n$ is a perfect $(1,d)$-matching in the graph $\mathfrak{T}(-\dot{v}_1,-\dot{v}_2)$, which in turn is a subgraph of the bipartite graph $\Gamma^{(n+1)}$, we can use Lemma \ref{access} for $u^{\perp}_j, X,\dot{v}_1$ and arrive at $u^{\perp}_j \xhookdoubleheadrightarrow{\mathfrak{M}^2_n,X} \dot{v}_1$. 
This gives us sequences $\{v_0',\ldots,v'_n\}, \{u_0',\ldots,u'_{n-1}\}$ as in Definition \ref{acc}. 
We now apply an argument similar to one from the proof of Lemma \ref{1stsim}. 
We set

\begin{align*}
M':= (\mathfrak{M}^2_n\setminus\{(u^{l+1}_n,\dot{v}^{l+1}_{n,1}),\ldots, (u^{l+1}_n,\dot{v}^{l+1}_{n,d}),(u'_{0},v'_{0}),\ldots ,(u'_{n-1},v'_{n-1})\} )\cup\\ \{(u^{\perp}_j,v^{\perp}_{j,1}),\ldots, (u^{\perp}_j,v^{\perp}_{j,d-1}), (u^{\perp}_j,v'_{0}), (u'_{0},v'_{1}),\ldots, (u'_{n-1},v'_{n})\}, 
\end{align*} 
where $v'_{n} =\dot{v}_1$. 
Observe that 
$$
(\mathfrak{U}\cup\{u_j^{\perp}\},(\mathfrak{V}\cup\{v_{j,1}^{\perp},\ldots, v_{j,d-1}^{\perp}\})\setminus \{\dot{v}_2\})=\mathfrak{T}(+u_j^{\perp})\setminus \{\dot{v}_2\}=\mathfrak{T}(+u_j^{\perp},-\dot{v}_2).$$ 
We remind the reader that $\mathfrak{M}^2_n$ is a perfect $(1,d)$-matching in the graph $\Gamma^{(n)\star}(-u_n)$. 
We have obtained $M'$ from $\mathfrak{M}^2_n$ by removing $d$ edges incident to $u^{l+1}_n$, adding $d$ edges incident to $u^{\perp}_j$, and the following replacement: 
for each $u'_i$ we replace one edge incident to it by another incident edge.
It follows that the matching $M'$ is a perfect $(1,d)$-matching in the graph $\mathfrak{T}(+u_j^{\perp},-\dot{v}_2)$. Therefore that graph satisfies Hall's $d$-harem condition. 

We define $\Gamma':=\mathfrak{T}(+u_j^{\perp})$, denote $\Gamma' = (U',V')$ and put 
$$\dot{\Gamma}^{(n+1)\perp}=(\dot{U}^{(n)\perp}\setminus \{u_j^{\perp}\},\dot{V}^{(n)\perp}\setminus \{v_{j,k}^{\perp}:1\leq k \leq d-1\}).$$
This ends the first part of the proof.

\noindent 
$\bullet$ \textit{Part 2.}
We check whether the graph $\Gamma'$ satisfies Hall's $d$-harem condition.
If it does, then by Part 1 $\Gamma'$ has to be equal to $\mathfrak{T}(+u_j^{\perp})$ and 
the proof is finished by the second option of the formulation. 
If it does not then repeating the reasoning of Part 1 we see that there exists a minimal connected set $X$ such that $\dot{v}_2\in X\subset V'$ and $|N_{\Gamma'}(X)|<\frac{1}{d}|X|$.  
Again, using Claim \ref{gn1} we obtain
$$N_{\Gamma^{(n+1)}}(X)\cap \dot{U}^{(n)\perp}\neq\emptyset,$$ 
i.e. there exists some $u^{\perp}_i\in N_{\Gamma^{(n+1)}}(X)$. We denote by $v^{\perp}_{i,1},\ldots, v^{\perp}_{i,d-1}$ the vertices adjacent to $u^{\perp}_j$ in $\dot{\Gamma}^{(n+1)\perp}$.

The matching $M'$ obtained in the first part of the proof is a perfect $(1,d)$-matching for either the graph $\mathfrak{T}(-\dot{v}_2)$, or the graph $\mathfrak{T}(+u_j^{\perp},-\dot{v}_2)$. 
Each of them is a subgraph of $\Gamma^{(n+1)}$. 
Therefore we can use Lemma \ref{access} for $u^{\perp}_i, X,\dot{v}_2$. 
We have $u^{\perp}_i \xhookdoubleheadrightarrow{M',X} \dot{v}_2$. 
Again we can apply the argument of Lemma \ref{1stsim} to obtain an appropriate matching: 
$$
M'':= (M'\setminus\{(u'_{0},v'_{0}),\ldots , (u'_{n-1},v'_{n-1}) \} )\cup$$
$$\{(u^{\perp}_i,v^{\perp}_{i,1}),\ldots, (u^{\perp}_i,v^{\perp}_{i,d-1}),(u^{\perp}_i,v'_{0}), (u'_{0},v'_{1}), \ldots,(u'_{n-1},v'_{n})\}.
$$
We have obtained $M''$ from $M'$ by adding $d$ edges incident to $u^{\perp}_i$, and replacing one edge in matching for each of $u'_k$ in such a way, that $\dot{v}_2$ is adjacent to one edge in $M''$.

The final argument depends on two possible outputs of Part 1.  
If the graph $\mathfrak{T}(-\dot{v}_2)$ does not satisfy  Hall's $d$-harem condition (i.e. $u^{\perp}_j$ is involved), then $M''$ is a perfect $(1,d)$-matching in the graph $\mathfrak{T}(+u_i^{\perp},+u_j^{\perp})$. 
If it does, we redefine $u^{\perp}_j:=u^{\perp}_i$ and then $M''$ becomes a perfect $(1,d)$-matching in the graph $\mathfrak{T}(+u_j^{\perp})$. 
Therefore if $\mathfrak{T}$ does not satisfy Hall's $d$-harem condition, then either $\mathfrak{T}(+u_j^{\perp})$ or $\mathfrak{T}(+u_i^{\perp},+u_j^{\perp})$ satisfies this condition.
\end{proof}

\section{Proof of the Main Theorem}

\begin{proof}[Proof of Theorem \ref{hhco}]
Let us apply the construction of Section \ref{simconst}.
This construction works modulo Claims \ref{cs1p1}, \ref{cs1p2part1}, \ref{csnp1}, \ref{cs1p2}. Claims \ref{cs1p1} and \ref{csnp1} follow from Lemma \ref{1stsim}. Claims \ref{cs1p2part1} and \ref{cs1p2} follow from Lemma \ref{2ndsim}.
Since for every $n$ the union $\bigcup\limits_{i=1}^n M_i$ consist of $(1,d-1)$-fans, the final set of edges $M$ is an $(1,d-1)$-matching.

For every $u\in U$ there is a step where an edge incident to $u$ is added to $M$. 
Furthermore, if the copy $v_u$ was not added to $M$ earlier, the construction ensures that $v_u$ is also added to $M$ in the second part of this step. 
This guarantees that $M$ is a perfect $(1,d-1)$-matching of the graph $\Gamma$. 
 
Until the end of the proof we will abuse the notation and use a natural number $n$ both for vertices and numbers of steps.

It remains to show that $f$ has controlled sizes of its cycle. 
To see condition $(i)$ of Definition \ref{cycles} note that the edges $(u_0, v^0_{0,1})$ and $(v_{u_0},u_{v^0_{0,1}})$ are added to $M$ at step 1, i.e. $f^2(u_0)=f(u_{v^0_{0,1}})=u_0$. 
Since $u_0=1$, condition $(i)$ is satisfied.

As a vertex $n$ appears in $M$ at the $n$-th step at latest, the length of a cycle created at the $n$-th step cannot be greater than $\max\{2,n\}$.  
In particular, if $n$ is in a cycle then $f^i(n)=n$ for some $i\leq n$ and condition $(ii)$ of Definition \ref{cycles} is satisfied. 

It remains to show that condition $(iii)$ is satisfied.
Let $\mathsf{f}_i$ be the partial function (living in $\mathbb{N}$) realized by $M_i$. 
Then $M_i = \{(\mathsf{f}_i(m),m):m\in\mathsf{Dom}(\mathsf{f}_i)\}$ is the graph of $\mathsf{f}_i$.  
We denote by $f_n$ the partial function  realized by $\bigcup\limits_{i=0}^{n}M_i$.
Let $F_n = \{(f_n(m),m):m\in\mathsf{Dom} (f_n )\}$ be the graph of $f_n$ on $\mathsf{Dom} (f_n )\cup \mathsf{Rng} (f_n )$. 
Then each $M_i$ with $i\le n$, is a subgraph of $F_n$. 

Since each subgraph $M_i$ is connected,  $F_n$ has at most $n+1$ connected components. 
Furthermore, one of the following properties holds:
\begin{itemize}
\item $M_n$ is a connected component of $F_n$ and contains a cycle;
\item $M_n$ is connected and there is a vertex of degree $1$ in $F_{n-1}$ which is an image of a vertex from $M_n$.  
\end{itemize} 
The first possibility arises in cases (2) and (3) of part 2 of the $(n+1)$-th step. 
The second possibility is the result of case (1) of this part of the step (see Section \ref{nplus1p2} and Remarks before part 2). 
Since each element is included into $M$ together with a fan, this property guarantees that $F_n$ consists only of vertices of degree $1$ and $d$.  
When a vertex has degree $1$, then its $F_n$-neighborhood is exactly its $f_n$-image.
When a vertex has degree $d$, then its $f_n$-image and $d-1$ preimages are in $F_n$.
In particular, each connected component of $F_n$ contains a cycle. 
The length of the cycle is not greater than $\max\{n,2\}$.

Since the value $f_{n-1} (n)$ is defined, $n$ belongs to some connected component of $f_{n-1}$. 
Thus, there exist $k$ and $l$ such that $f_{n-1}^{k+l}(n)=f_{n-1}^k(n)$.
These $k$ and $l$ work for the equality $f^{k+l}(n)=f^k(n)$. 
It remains to show that $k$ can be bounded by $2n$ and $l$ by $n$. 

The latter estimate is easy: the biggest cycle that can be constructed before the $n$-th step, is not longer than $\max\{n-1,2\}$. 
Below we will only use the inequality $l \le n$ for simplicity. 

In order to show that $k$ is bounded by $2n$, let us estimate the size of a subset of $U$ that can be added to the matching $M$ in the process of the $n$-th step of the construction, in the situation when a new cycle is not created. 
It must consist of elements of $U^{(n-1)\perp}$ added to the matching at the iteration of part $2$ of the $n$-th step together with $u_n$.
Therefore, we can bound it by the maximal possible number $|U^{(n-1)\perp}|+1$.

Let us denote the number of elements from $U^{(s-1)\perp}$ added to $M$ at the $s$-th step by $\ell_s$.
Then for each $m\in M_s$ we have $f^{\ell_s +1}(m)\in M_j$ where $j \leq s-1$. 
If no cycle is constructed at the $j$-th step then $f^{\ell_{j}+1}(f^{\ell_{s}+1}(m))\in M_{i}$ for some $i\leq j-1$. 
Iterating this argument and applying it to $m:=n$ we arrive at 
$$
k\leq \sum\limits_{s=1}^n (\ell_s+1).
$$ 
Since at the $n$-th step the value $|\bigcup\limits_{s=1}^{n-1} U^{(s)\perp}|$ does not exceed $n-1$, 
$$
\sum\limits_{s=1}^n \ell_s \le n-1.
$$
We see that $k\leq 2n-1$. 
Thus condition $(iii)$ of Definition \ref{cycles} is satisfied.
\end{proof}

\printbibliography

\end{document}